\documentclass[11pt]{article}
\usepackage[latin9]{inputenc}
\usepackage{geometry}
\usepackage{color}
\usepackage{mathrsfs}
\usepackage{amsmath}
\usepackage{amsthm}
\usepackage{amssymb}
\usepackage{setspace}
\usepackage[unicode=true,pdfusetitle,
 bookmarks=true,bookmarksnumbered=false,bookmarksopen=false,
 breaklinks=false,pdfborder={0 0 1},backref=page,colorlinks=true]
 {hyperref}
\hypersetup{
 linkcolor=blue}

\makeatletter
\numberwithin{equation}{section}
\numberwithin{figure}{section}
\theoremstyle{plain}
\newtheorem{thm}{\protect\theoremname}[section]
\theoremstyle{plain}
\newtheorem{cor}[thm]{\protect\corollaryname}
\theoremstyle{plain}
\newtheorem{lem}[thm]{\protect\lemmaname}
\theoremstyle{plain}
\newtheorem{prop}[thm]{\protect\propositionname}
\theoremstyle{definition}
\newtheorem{defn}[thm]{\protect\definitionname}
\theoremstyle{remark}
\newtheorem{rem}[thm]{\protect\remarkname}

\usepackage{graphicx}
\usepackage{appendix}

\usepackage{enumitem}
\usepackage[backref=page]{hyperref}

\makeatother

\providecommand{\corollaryname}{Corollary}
\providecommand{\definitionname}{Definition}
\providecommand{\lemmaname}{Lemma}
\providecommand{\propositionname}{Proposition}
\providecommand{\remarkname}{Remark}
\providecommand{\theoremname}{Theorem}

\begin{document}
\global\long\def\F{\mathrm{\mathbf{F}} }%
\global\long\def\Aut{\mathrm{Aut}}%
\global\long\def\C{\mathbf{C}}%
\global\long\def\H{\mathcal{H}}%
\global\long\def\U{\mathbf{U}}%
\global\long\def\P{\mathcal{P}}%
\global\long\def\ext{\mathrm{ext}}%
\global\long\def\hull{\mathrm{hull}}%
\global\long\def\triv{\mathrm{triv}}%
\global\long\def\Hom{\mathrm{Hom}}%

\global\long\def\trace{\mathrm{tr}}%
\global\long\def\End{\mathrm{End}}%

\global\long\def\L{\mathcal{L}}%
\global\long\def\W{\mathcal{W}}%
\global\long\def\E{\mathbb{E}}%
\global\long\def\SL{\mathrm{SL}}%
\global\long\def\R{\mathbf{R}}%
\global\long\def\Pairs{\mathrm{PowerPairs}}%
\global\long\def\Z{\mathbf{Z}}%
\global\long\def\rs{\to}%
\global\long\def\A{\mathcal{A}}%
\global\long\def\a{\mathbf{a}}%
\global\long\def\rsa{\rightsquigarrow}%
\global\long\def\D{\mathbf{D}}%
\global\long\def\b{\mathbf{b}}%
\global\long\def\df{\mathrm{def}}%
\global\long\def\eqdf{\stackrel{\df}{=}}%
\global\long\def\ZZ{\overline{Z}}%
\global\long\def\Tr{\mathrm{Tr}}%
\global\long\def\N{\mathbf{N}}%
\global\long\def\std{\mathrm{std}}%
\global\long\def\HS{\mathrm{H.S.}}%
\global\long\def\e{\mathbf{e}}%
\global\long\def\c{\mathbf{c}}%
\global\long\def\d{\mathbf{d}}%
\global\long\def\AA{\mathbf{A}}%
\global\long\def\BB{\mathbf{B}}%
\global\long\def\u{\mathbf{u}}%
\global\long\def\v{\mathbf{v}}%
\global\long\def\spec{\mathrm{spec}}%
\global\long\def\Ind{\mathrm{Ind}}%
\global\long\def\half{\frac{1}{2}}%
\global\long\def\Re{\mathrm{Re}}%
\global\long\def\Im{\mathrm{Im}}%
\global\long\def\Rect{\mathrm{Rect}}%
\global\long\def\Crit{\mathrm{Crit}}%
\global\long\def\Stab{\mathrm{Stab}}%
\global\long\def\SL{\mathrm{SL}}%
\global\long\def\TF{\mathsf{TF}}%
\global\long\def\p{\mathfrak{p}}%
\global\long\def\j{\mathbf{j}}%
\global\long\def\uB{\underline{B}}%
\global\long\def\tr{\mathrm{tr}}%
\global\long\def\rank{\mathrm{rank}}%
\global\long\def\K{\mathcal{K}}%
\global\long\def\hh{\mathbb{H}}%

\global\long\def\EE{\mathcal{E}}%
\global\long\def\PSL{\mathrm{PSL}}%

\title{Extension of Alon's and Friedman's conjectures to \linebreak{}
Schottky surfaces}
\author{Michael Magee and Frédéric Naud}
\maketitle
\begin{abstract}
Let $X=\Lambda\backslash\hh$ be a Schottky surface, that is, a conformally
compact hyperbolic surface of infinite area. Let $\delta$ denote
the Hausdorff dimension of the limit set of $\Lambda$.

We prove that for any compact subset $\K\subset\{\,s\,:\,\Re(s)>\frac{\delta}{2}\,\}$,
if one picks a random degree $n$ cover $X_{n}$ of $X$ uniformly
at random, then with probability tending to one as $n\to\infty$,
there are no resonances of $X_{n}$ in $\K$ other than those already
belonging to $X$ (and with the same multiplicity). This result is
conjectured to be the optimal one for bounded frequency resonances
and is analogous to both Alon's and Friedman's conjectures for random
graphs, which are now theorems due to Friedman and Bordenave-Collins,
respectively.
\end{abstract}
\tableofcontents{}

\section{Introduction}

Let $X$ be an infinite area surface equipped with a metric of constant
curvature -1, with a finitely generated non-abelian fundamental group
and without cusps. Such a surface is called a \emph{conformally compact
hyperbolic surface}. 

This paper addresses the question of whether typical such surfaces
have (almost) optimal spectral gaps. Here, a spectral gap refers either
to a gap in the spectrum of the Laplace-Beltrami operator or, more
ambitiously, to a region where there are no resonances of the meromorphically
continued resolvent. We give the background to this spectral theory
now; for general background on the spectral theory of infinite area
hyperbolic surfaces the reader should see \cite{Borthwick}.

One can obtain $X$ as a quotient $X=\Lambda\backslash\mathbb{H}$
where $\mathbb{H}$ is the hyperbolic upper half plane and $\Lambda$
is a discrete, finitely generated, non-abelian free subgroup of $\PSL_{2}(\R)$,
and, in fact, by a result of Button \cite{Button} there are generators
$\gamma_{1}\ldots,\gamma_{d}$ of $\Lambda$ for some $d\geq2$ which
form a Schottky system in the sense of $\S\S$\ref{subsec:Boundary-coding-of}.
The orbit of $\Lambda$ on any fixed point of $\mathbb{H}$ accumulates
on $\partial\mathbb{H}$ in a set called the \emph{limit set }of $\Lambda$;
we write $\delta=\delta(\Lambda)$ for the Hausdorff dimension of
the limit set. The assumption that $\Lambda$ is non-abelian and $X$
is infinite area implies $\delta\in(0,1)$.

Let $\Delta_{X}$ denote the Laplace-Beltrami operator on $L^{2}(X)$.
The spectrum of $\Delta_{X}$ in the range $[\frac{1}{4},\infty)$
is always continuous with no embedded eigenvalues, and the spectrum
is always discrete below $\frac{1}{4}$, by work of Lax-Phillips \cite{LP}.
If $\delta>\frac{1}{2}$ then the bottom of the spectrum occurs at
$\delta(1-\delta)$ by a result of Patterson \cite{Patterson}. Therefore
a naive definition of spectral gap in this case is $\lambda_{1}-\delta(1-\delta)$
where $\lambda_{1}$ is the minimum element of the spectrum other
than $\delta(1-\delta$), including $\delta(1-\delta)$ itself if
it occurs with multiplicity larger than one. However, if $\delta\leq\frac{1}{2}$
then the spectrum is precisely $[\frac{1}{4},\infty)$ and this notion
of spectral gap has no meaning. Moreover, even if $\delta>\frac{1}{2}$,
it does not give the strongest possible information as we explain
now.

The resolvent operator
\[
R_{X}(s)\eqdf(\Delta_{X}-s(1-s))^{-1}:C_{0}^{\infty}(X)\rightarrow C^{\infty}(X)
\]
 has meromorphic continuation from $\Re(s)>\frac{1}{2}$ to the entire
complex plane \cite{MazzeoMelrose}. The poles of this family of operators
are called \emph{resonances }of $X$. The multiplicity of a resonance
$s$ is given by $\mathrm{rank}\left(\int_{\gamma}R_{Y}(s)ds\right)$
where $\gamma$ is an anticlockwise oriented circle enclosing $s$
and no other resonance of $X$. Resonances $s$ with $\Re(s)>\frac{1}{2}$
yield eigenvalues $s(1-s)$ with the same multiplicity. We write $\mathcal{R}_{X}$
for the multiset of resonances of $X$, including multiplicities.

Now it is apparent that beyond $L^{2}$ spectral gaps, we can ask
for resonance-free regions of the complex plane. We summarize what
is known in this regard assuming that $X$ is connected (otherwise
there is no spectral gap in any sense). In the half-plane $\left\{ \Re(s)>\half\right\} $,
there are finitely many resonances that all lie on the real line,
the right-most one being at $s=\delta$. If $\delta\leq\half$, then
we know from \cite{NaudExpanding} that there exists $\epsilon(\Lambda)>0$
such that in the half plane $\left\{ \Re(s)>\delta-\epsilon\right\} $,
the only resonance is at $s=\delta$. Therefore the spectral gap is
then defined as the maximal size of this $\epsilon(\Lambda)$. It
was shown in \cite{JNS} that this gap can be arbitrarily small.

A conjecture of Jakobson and Naud \cite[Conj. 2]{JN1} predicts (among
other things) that for any $\epsilon>0$ there are infinitely many
resonances with $\Re(s)>\frac{\delta}{2}-\epsilon$ and hence one
does not expect to obtain spectral gap larger than $\frac{\delta}{2}$.
It is a pressing question as to whether `typical' surfaces have a
spectral gap close to this optimal value; this question has famous
analogs in graph theory that we will turn to shortly.

In \cite{MN1} we introduced a model of random conformally compact
hyperbolic manifolds based on random covering spaces. For any $n\in\N$
the collection of degree $n$ Riemannian covering spaces of $X$ is
a finite set and hence we can pick one of these covering spaces uniformly
at random. Note that if $X'$ covers $X$, then any resonance of $X$
is a resonance of $X'$, with at least as large multiplicity. The
main theorem of this paper is the following.
\begin{thm}
\label{thm:main-theorem}Let $X_{n}$ denote a uniformly random degree
$n$ covering space of $X$. For any compact set $\K\subset\{\,s\,:\,\Re(s)>\frac{\delta}{2}\,\}$,
with probability tending to one as $n\to\infty$
\[
\mathcal{R}_{X_{n}}\cap\K=\mathcal{R}_{X}\cap\K.
\]
\end{thm}

This theorem says that for any $\epsilon>0$, the resonance set of
random $X_{n}$ in the region $\{\,s\,:\,\Re(s)>\frac{\delta}{2}+\epsilon\}$
is almost surely optimal \emph{provided we restrict to bounded frequency
(imaginary part). }A weaker version of Theorem \ref{thm:main-theorem}
with $\frac{\delta}{2}$ replaced by $\frac{3\delta}{4}$ was proved
in \cite[Thm. 1.1]{MN1}

Theorem \ref{thm:main-theorem} together with the main result of \cite{BMM}
implies the following corollary on $L^{2}$ eigenvalues.
\begin{cor}
Assume that $\delta>\half$, and that the base surface\footnote{In this case $X$ is either a pair of pants with three funnels or
a torus with one funnel.} $X$ has Euler characteristic $-1$. Let $X_{n}$ denote a uniformly
random degree $n$ covering space of $X$. With probability tending
to one as $n\to\infty$ the only eigenvalue of $\Delta_{X_{n}}$ is
$\delta(1-\delta)$.
\end{cor}

We now explain the analogy with random graphs; this is also discussed
in detail in \cite[Introduction]{MN1}. A celebrated conjecture of
Alon \cite{Alon} predicted that for any $\epsilon>0$, a uniformly
random $d$-regular graph on $n$ vertices, with probability tending
to one as $n\to\infty$, has no eigenvalues of its adjacency operator
larger than $2\sqrt{d-1}+\epsilon$, other than $d$. The relevance
of the value $2\sqrt{d-1}$ is both that it is the spectral radius
of the adjacency operator on the universal cover of a $d$-regular
graph (a $d$-regular tree), and also, a result of Alon-Boppana \cite{Nilli}
states that any sequence of $d$-regular graphs on $n$ vertices has
second largest eigenvalue of their adjacency operators at least $2\sqrt{d-1}-o(1)$
as $n\to\infty$. Hence the value $2\sqrt{d-1}$ is analogous to $\frac{\delta}{2}$
here. It is interesting however that even though the Alon-Boppana
bound is not very hard to prove, the analog of this fact in the current
setting is still only a conjecture. Alon's conjecture was proved by
Friedman in \cite{Friedman} (see also \cite{bordenave2015new} for
a new proof and \cite{PUDER} for a proof of only slightly weaker
result).

Friedman conjectured in \cite{FriedmanRelative} that a variant of
Alon's conjecture should hold for random degree $n$ covering spaces\footnote{In graph-theoretic literature, these are called $n$-lifts.}
of a fixed finite base graph provided
\begin{itemize}
\item One replaces $2\sqrt{d-1}$ by the spectral radius of the adjacency
operator on the universal cover of the graph, and 
\item One allow eigenvalues of the adjacency operator that already belonged
to the base graph (as one must).
\end{itemize}
Friedman's conjecture was proved in a breakthrough work of Bordenave
and Collins \cite{BordenaveCollins}. In fact Bordenave-Collins proved
a vast generalization of Friedman's conjecture where one twists a
random Hecke operator, formed from random permutations, by arbitrary
fixed finite dimensional matrices, assuming the matrices satisfy a
symmetry condition that forces the resulting operator to be self-adjoint.
The work of Bordenave-Collins is a vital ingredient of the current
work.

The first result on spectral gap of random hyperbolic surfaces is
due to Brooks and Makover \cite{BrooksMakover} who prove that for
a combinatorial model of random closed hyperbolic surface, depending
on a parameter $n$ that influences the genus (non-deterministically),
that there exists a constant $C>0$ such that with probability tending
to one as $n\to\infty$, the first non-zero eigenvalue $(\lambda_{1})$
of the Laplacian satisfies $\lambda_{1}\geq C$. Mirzakhani proved
in \cite{MirzakhaniRandom} that for Weil-Petersson random closed
hyperbolic surfaces of genus $g$, with probability tending to one
as $g\to\infty$ one has $\lambda_{1}\geq0.0024$. In \cite{MN1}
the authors of the current paper proved that Theorem \ref{thm:main-theorem}
holds with $\frac{\delta}{2}$ replaced by $\frac{3\delta}{4}$.

Returning to closed surfaces, by building on \cite{MPasympcover},
the authors and Puder proved in \cite{MageeNaudPuder} that for a
uniformly random degree $n$ cover of a fixed closed hyperbolic surface,
there are for any $\epsilon>0$, with probability tending to one as
$n\to\infty$, no new eigenvalues of the covering space below $\frac{3}{16}-\epsilon$.
This result was adapted to Weil-Petersson random surfaces independently
by Wu and Xue \cite{wu2021random} and Lipnowski and Wright \cite{lipnowski2021optimal};
here the corresponding statement is that there are no eigenvalues
between $0$ and $\frac{3}{16}-\epsilon$. These `$\frac{3}{16}$'
results are, when it comes \emph{solely} to $L^{2}$ eigenvalues,
at the strength of the result of \cite{MN1} giving resonance-free
regions in terms of $\frac{3\delta}{4}$; for compact surfaces $\delta=1$
and the eigenvalue is written $\lambda=s(1-s)$ ($\frac{3}{4}\left(1-\frac{3}{4}\right)=\frac{3}{16})$.
On the other hand, closed surfaces involve additional difficulties
due to their non-free fundamental groups.

There are other related works on Weil-Petersson random surfaces that
do not imply spectral gaps, but instead offer some spectral delocalization
results \cite{Monk,GLST}.

Uniform spectral gap for (deterministic) covering spaces of infinite
area hyperbolic surfaces has also been of interest in number theoretic
settings; see \cite{BGS2,OW,MOW,Gamburd1} for a selection of results;
the single quantitative result here is by Gamburd \cite{Gamburd1}.
Much of the motivation of these spectral gap results came from the
`thin groups' research program; see Sarnak's article \cite{Sarnak}
for an overview. 

Another closely related concept is that of \emph{essential spectral
gap}; referring to a half-plane where only finitely many resonances
appear. Two important results here are due to Bourgain and Dyatlov.
Let $X$ be conformally compact as above. The first result, proved
in \cite{BDFourier}, says that there is $\epsilon>0$, \emph{depending
only on $\delta$, }such that there are only finitely many resonances
in $\{\,s\,:\,\Re(s)>\delta-\epsilon\,\}$. This result is relevant
if $\delta\leq1/2$. On the other hand, it is proved in \cite{BDgap}
that there are only finitely many resonances in $\{\,s\,:\,\Re(s)>\frac{1}{2}-\eta\,\}$
for some $\eta=\eta(\Lambda)>0$, this result being relevant if $\delta>\half$.
A conjecture of Jakobson-Naud \cite{JN1} says that the optimal essential
spectral gap corresponds to finitely many resonances in $\left\{ \Re(s)>\frac{\delta}{2}+\varepsilon\right\} $
for any $\varepsilon>0$. We do not know yet if our probabilistic
techniques can be used to address high frequency problems (i.e. resonances
with large imaginary parts) and have not attempted to do so in the
present paper, however the present paper says that in some sense this
conjecture holds in the bounded frequency, large cover regime. For
a broader perspective on resonances of hyperbolic surfaces than we
are able to offer here, the reader can see Zworski's survey article
\cite{Zworski_survey}.

\subsection*{Acknowledgments}

We thank Charles Bordenave, Benoît Collins, and Doron Puder for helpful
conversations related to this work.

This project has received funding from the European Research Council
(ERC) under the European Union\textquoteright s Horizon 2020 research
and innovation programme (grant agreement No 949143).

\global\long\def\i{\mathbf{i}}%

\section{Preliminaries}

\subsection{Boundary coding of Schottky groups and Bowen's formula\label{subsec:Boundary-coding-of}}

In this section, we set some notations (which are essentially similar
to \cite{MN1} and \cite{BDFourier} with some slight changes).

Let $d\geq2$ and $\mathcal{I}=\{1,\ldots,2d\}$. If $i\in\mathcal{I}$,
then we write $\bar{i}\eqdf i+d\bmod2d$ so $i\in\mathcal{I}$. For
each $i\in\mathcal{I}$we are given an \emph{open} disc $D_{i}$ in
$\C$ with center in $\R$. The closures of the discs $D_{i}$ for
$i\in\mathcal{I}$ are assumed to be disjoint from one another. We
let $I_{i}\eqdf D_{i}\cap\R$, an open interval. We write $\D\eqdf\cup_{i\in\mathcal{I}}D_{i}$
for the union of the discs.

We consider the usual action of $\SL_{2}(\R)$ by Möbius transformations
on the extended complex plane $\hat{\C}=\C\cup\{\infty\}$. We are
given for each $i\in\mathcal{I}$ a matrix $\gamma_{i}\in\SL_{2}(\R)$
with the mapping property
\[
\gamma_{i}\left(\hat{\C}\backslash D_{\bar{i}}\right)=\overline{D_{i}},\quad\gamma_{\bar{i}}=\gamma_{i}^{-1}.
\]
We recall that we denote $\Lambda\eqdf\langle\,\gamma_{i}\,:\,i\in\mathcal{I\,}\rangle$
for the group generated by the $\gamma_{i}$. Since the discs $D_{i}$
are disjoint, Klein's Ping-Pong Lemma shows that $\Lambda$ is a \textit{free
subgroup} of $\SL_{2}(\R)$. The converse is actually true in dimension
$2$: every conformally compact hyperbolic surface $X$ can be uniformized
by a Schottky group $\Lambda$ so that $X=\Lambda\backslash\hh$,
see \cite{Button}. The elements of $\Lambda$ can be encoded by words
in the alphabet $\mathcal{I}$ as follows. A \emph{word }is a finite
sequence 
\[
\i=(i_{1},\ldots,i_{n}),\quad n\in\N\cup\{0\}
\]
such that $i_{j}\neq\overline{i_{j+1}}$ for $j=1,\ldots,n-1$. We
say that $n$ is the\emph{ length} of $\i$ and denote this by $|\i|=n$.
We write $\W$ for the collection of all words, $\W_{N}$ for the
words of length $N$, and $\W_{\geq N}$ for the words of length $\geq N$.
We write $\emptyset$ for the empty word and write $\W^{\circ}=\W-\{\emptyset\}$.
For $\i=(i_{1},\ldots,i_{n}),\ \j=(j_{1},\ldots j_{m})\in\W$ we write 
\begin{itemize}
\item ${\bf \i'}\eqdf(i_{1},\ldots,i_{n-1})$ if $\i=(i_{1},\ldots,i_{n})$
and $n\geq1$. 
\item $\hat{\i}\eqdf(i_{2},\ldots,i_{n})$ if $\i=(i_{1},\ldots,i_{n})$
and $n\geq1$.
\item $\i\to\j$ if either of $\i$ or $\j$ is empty, or else $i_{n}\neq\overline{j_{1}}$,
in which case $(i_{1},\ldots,i_{n},j_{1},\ldots,j_{m})$ is in $\W^{\circ}$
and we write $\i\j$ for this concatenation. 
\end{itemize}
If $\i=(i_{1},\ldots,i_{n})\in\W$ then we associate to $\i$ the
group element $\gamma_{\i}\eqdf\gamma_{i_{1}}\ldots\gamma_{i_{n}}$;
here $\gamma_{\emptyset}=\mathrm{id}.$ The map $\i\in\W\mapsto\gamma_{\i}\in\Gamma$
is a one-to-one encoding of $\Gamma$. We write $\overline{\i}\eqdf(\overline{i_{n}},\ldots,\overline{i_{1}})$
and call this the \emph{mirror} \textit{word} of $\i$. Note that
$\gamma_{\overline{\i}}=\gamma_{\i}^{-1}$. If $\i=(\i_{1},\ldots,\i_{n})\in\W^{\circ}$
we set 
\[
D_{\i}=\gamma_{\i'}(D_{i_{n}}),\quad I_{\i}=\gamma_{\i'}(I_{i_{n}})
\]
and write $|I_{\i}|$ for the length of the open interval $I_{\i}$.
We view this set up as fixed henceforth, so all constants will depend
on $\Lambda$.

The \emph{Bowen-Series map} $T:\D\to\hat{\C}$ is given by 
\[
T\lvert_{D_{i}}=\gamma_{i}^{-1}=\gamma_{\bar{i}}.
\]
The Bowen-Series map is eventually expanding \cite[Prop. 15.5]{Borthwick};
this will be made explicit below so we do not give the general definition
now. The limit set $K=K(\Lambda)$ of $\Lambda$, defined in the Introduction,
coincides with the non-wandering set of $T$: 
\[
K(\Lambda)\eqdf\bigcap_{n=1}^{\infty}T^{-n}(\D).
\]
The limit set $K$ is a compact $T$-invariant subset of $\R$. Given
a Hölder continuous map $\varphi:K\rightarrow\R$, the \emph{topological
pressure} $P(\varphi)$ can be defined through the variational formula:
\[
P(\varphi)\eqdf\sup_{\mu}\left(h_{\mu}(T)+\int_{\Lambda}\varphi d\mu\right),
\]
where the supremum is taken over all $T$-invariant probability measures
on $K(\Lambda)$, and $h_{\mu}(T)$ stands for the measure-theoretic
entropy. A famous result of Bowen \cite{Bowen} says that the map
$\R\to\R$,
\[
r\mapsto P(-r\log\vert T'\vert)
\]
is convex\footnote{Convexity follows obviously from the variational formula above.},
strictly decreasing and vanishes exactly at $r=\delta\eqdf\delta(\Lambda)$,
the Hausdorff dimension of the limit set $K(\Lambda)$. In addition,
it is not difficult to see from the variational formula that $P(-r\log\vert T'\vert)$
tends to $-\infty$ as $r\rightarrow+\infty$. For simplicity, we
will use the notation $P(r)$ in place of $P(-r\log\vert T'\vert)$.
The pressure functional plays an important role in the sequel.

\subsection{Random covering spaces, permutations, and representations\label{subsec:Random-covering-spaces,}}

Let $S_{n}$ denote the group of permutations of $[n]\eqdf\{1,\ldots,n\}$.
Our random covering spaces of $X$ are parameterized by 
\[
\phi\in\Hom(\Lambda,S_{n}).
\]
Of course, choosing $\phi$ is the same as choosing 
\[
\sigma=(\sigma_{1},\ldots,\sigma_{d})\in S_{n}^{d}
\]
where $\sigma_{i}\eqdf\phi(\gamma_{i})$. We view $\sigma$ and $\phi$
as coupled in this way in the rest of the paper.

Given $\phi\in\Hom(\Lambda,S_{n})$, we construct a covering space
of $X$ as follows. Let $\Lambda$ act on $\hh\times[n]$ by 
\[
\gamma(z,i)=(\gamma.z,\phi(\gamma)[i])
\]
 where the action on the first factor is by Möbius transformation.
The quotient
\[
X_{\phi}\eqdf\Lambda\backslash_{\phi}\left(\hh\times[n]\right)
\]
 is a degree $n$ cover of $X$. Choosing $\phi$ uniformly at random
in $\Hom(\Lambda,S_{n})$, the resulting $X_{\phi}$ is a uniformly
random degree $n$ Riemannian cover of $X$. 

Let $V_{n}\eqdf\ell^{2}([n])$. Given $\phi\in\Hom(\Lambda,S_{n})$
as above, we let $\rho_{n}:\Lambda\to\End(V_{n})$ be the representation
obtained by composing $\phi$ with the standard permutation representation
\[
S_{n}\to\End(V_{n}).
\]
We let $\rho_{n}^{0}$ denote the restriction of $\rho_{n}$ to the
space $V_{n}^{0}$ of functions on $[n]$ that are orthogonal to constants.
The parameters $\phi,\rho_{n}$, and $\rho_{n}^{0}$ (as well as $\sigma$)
are now coupled for the rest of the paper.

\subsection{Transfer operator and function spaces\label{subsec:Transfer-operator-and}}

Let $V$ be a finite dimensional complex Hilbert space. We consider
the Bergman space $\H(\D;V)$ that is the space of $V$-valued holomorphic
functions on $\D$ with finite norm with respect to the given inner
product 
\[
\langle f,g\rangle\stackrel{\df}{=}\int_{\D}\langle f(x),g(x)\rangle_{V}dm(x).
\]
Here $dm$ is Lebesgue measure on $\D$. This splits as an orthogonal
direct sum 
\[
\H(\D;V)=\bigoplus_{i\in\mathcal{I}}\H(D_{i};V).
\]
If $\{\e_{k}\}_{k=1}^{\infty}$ is any orthonormal basis of $\H(D_{i};\C)$,
and $x_{1},x_{2}\in D_{i}$, then the sum 
\[
\sum_{k=1}^{\infty}\e_{k}(x_{1})\overline{\e_{k}(x_{2})}\eqdf B_{D_{i}}(x_{1},x_{2})
\]
converges uniformly on every compact subset of $D_{i}$ and is called
the \emph{Bergman kernel of $D_{i}$. }It is given by the explicit
formula (cf. \cite[pg. 378]{Borthwick}) 
\begin{equation}
B_{D_{i}}(x_{1},x_{2})=\frac{r_{i}^{2}}{\pi\left[r_{i}^{2}-(\overline{x}_{2}-c_{i})(x_{1}-c_{i})\right]^{2}}\label{eq:explicit-bergman}
\end{equation}
where $r_{i},c_{i}$ are the radius and center of $D_{i}$. We let
\[
B_{\D}(z,w)\eqdf\sum_{i\in\mathcal{I}}\mathbf{1}\{z,w\in D_{i}\}B_{D_{i}}(z,w).
\]
 An important fact that will be used several times in this paper is
the reproducing property of the Bergman kernel: if $f\in\H(\D,V)$,
then we have for all $z\in\D$, 
\[
f(z)=\int_{\D}B_{\D}(z,w)f(w)dm(w),
\]
in particular for all compact subsets $K\subset\D$, the above explicit
formulas for the kernels show that there exists $C_{K}>0$ such that
\[
\sup_{z\in K}\Vert f(z)\Vert_{V}\leq C_{K}\Vert f\Vert_{\H(\D,V)}.
\]
Throughout the sequel, $\rho:\Lambda\to\U(V)$ will denote a finite
dimensional unitary representation of the Schottky group $\Lambda$.
The transfer operator acting on $\H(\D,V)$ is defined as follows:
\[
\L_{s,\rho}[f](x)f\eqdf\sum_{\substack{i\in\mathcal{I}\\
i\to j
}
}\gamma'_{i}(x)^{s}\rho(\gamma_{i}^{-1})f(\gamma_{i}(x))\quad x\in D_{j},j\in\mathcal{I}.
\]
Here $(\gamma_{i}'(x))^{s}$ is understood as $(\gamma_{i}'(x))^{s}:=e^{-s\tau(\gamma_{i}(x))}$,
where $\tau(z)$ is the analytic continuation to $\cup_{\vert\i\vert=2}D_{\i}$
of $\tau(x)=\log\vert T'(x)\vert$, and where $T$ is, as defined
above, the Bowen-Series map. In addition, we can assume that $\tau$
is also continuous on $\overline{\cup_{\vert\i\vert=2}D_{\i}}$ .

The following result appears in \cite[Lemma 15.7]{Borthwick} in the
case $V=\C$, and easily extends to the case of finite dimensional
$V$.
\begin{lem}
The operator $\L_{s,\rho}$ is a trace class operator on the Hilbert
space $\H(\D,V)$.
\end{lem}

Our interest in considering these operators stems from setting $\rho=\rho_{n}^{0}$
from $\S\S$\ref{subsec:Random-covering-spaces,}. The following fact
follows from \cite[Thm 2.2. and Prop. 4.4(2)]{MN1}.
\begin{prop}
\label{prop:new-resonances-and-evalue-1}If $s$ is a resonance of
$X_{\phi}$ that appears with greater multiplicity than in $X$, for
example, by not being a resonance of $X$, then 
\[
\det(1-\L_{s,\rho_{n}^{0}})=0
\]
and hence 1 is an eigenvalue of $\L_{s,\rho_{n}^{0}}$. (The determinant
above is a Fredholm determinant.)
\end{prop}

\subsection{Deterministic a priori bounds\label{subsec:Deterministic-a-priori-bounds}}

We will now introduce a useful notation from \cite{MN1}. In the rest
of the paper, for any $\i\in\W^{\circ}$, we define 
\[
\Upsilon_{\i}\eqdf|I_{\i}|,
\]
that is the length of the interval $I_{\i}$ associated to the word
$\i$. For all $\i\in\W_{\geq2}$, we have obviously
\begin{equation}
\Upsilon_{\i}\leq\Upsilon_{\i'}\label{eq:trivial-child-ineq}
\end{equation}
since $I_{i}\subset I_{i'}$. Therefore there exists trivially $c=c(\Lambda)>0$
such that for any $\i\in\W^{o}$,
\begin{equation}
0<\Upsilon_{\i}\leq c.\label{eq:upsilon-bounded}
\end{equation}

\begin{lem}
\label{lem:coarse-homomorphism}There exists a constant $K=K(\Lambda)>1$
such that the following bounds hold.
\begin{description}
\item [{(Rough~multiplicativity)}] For all $\i,\j\in\W^{\circ}$ with
$\i\rightarrow\j$, we have
\begin{equation}
K^{-1}\Upsilon_{\i}\Upsilon_{\j}\leq\Upsilon_{\i\j}\leq K\Upsilon_{\i}\Upsilon_{\j}.\label{eq:upsilon-mult-eq-1}
\end{equation}
\item [{(Mirror~estimate)}] For all $\j\in\W^{\circ}$, we have
\begin{equation}
K^{-1}\Upsilon_{\overline{\j}}\leq\Upsilon_{\j}\leq K\Upsilon_{\overline{\j}}.\label{eq:upsilon-mirror-eq}
\end{equation}
\item [{(Derivatives)}] For all $x\in D_{i_{\vert\i\vert}}$, we have
\begin{equation}
K^{-1}\Upsilon_{\i}\leq|\gamma'_{\i'}(x)|\leq K\Upsilon_{\i}.\label{eq:upsilon-deriv-eq}
\end{equation}
\item [{(Exponential~Bound)}] There is a constant $D=D(\Gamma)>1$ such
that 
\begin{equation}
\Upsilon_{\i}\leq KD^{-|\i|}.\label{eq:upsilon-exp-bound}
\end{equation}
\end{description}
\end{lem}

The proof of these bounds follow from \cite[Lemmas 3.1, 3.2, 3.4, 3.5]{MN1}
and follows previous estimates from \cite{BDFourier}. The following
result is \cite[Lemma 3.10]{MN1}.
\begin{lem}
\label{lem:Pressure-estimate}For all $r_{1},Q\in\R$ such that $0\leq r_{1}<Q$
there is a constant $C=C(r_{1},Q)>0$ such that for all $N\in\N_{0}$
and $r\in[r_{1},Q]$ we have
\begin{equation}
\sum_{\i\in\W_{N}}\Upsilon_{\i}^{r}\leq C\exp(NP(r_{1})).\label{eq:pressure-est-2}
\end{equation}
\end{lem}

\begin{lem}
\label{lem:norm-of-variation}Suppose that $\K\subset\C$ is compact.
There is a constant $C=C(\K)>0$ such that for all $\ell\in\N$ and
$s,s_{0}\in\K$
\[
\|\L_{s,\rho_{n}^{0}}^{\ell}-\L_{s_{0},\rho_{n}^{0}}^{\ell}\|\leq|s-s_{0}|C^{\ell},
\]
where $\Vert.\Vert$ denotes the operator norm on $\H(\D,V_{n}^{0})$.
\end{lem}

\begin{proof}
Let $f\in\H(\D,V_{n}^{0})$. We write 
\[
\Vert\L_{s,\rho_{n}^{0}}^{\ell}(f)-\L_{s_{0},\rho_{n}^{0}}^{\ell}(f)\Vert^{2}=\int_{\D}\Vert\L_{s,\rho_{n}^{0}}^{\ell}(f)-\L_{s_{0},\rho_{n}^{0}}^{\ell}(f)\Vert_{V_{n}^{0}}^{2}dm
\]
\[
=\sum_{j=1}^{2d}\int_{D_{j}}\left\Vert \sum_{\i\rightarrow j,\,|\i|=\ell}\left(e^{-s\tau^{(\ell)}(\gamma_{\i}z)}-e^{-s_{0}\tau^{(\ell)}(\gamma_{\i}z)}\right)\rho_{n}^{0}(\gamma_{\i}^{-1})f(\gamma_{\i}z)\right\Vert _{V_{n}^{0}}^{2}dm(z),
\]
where we have set
\[
\tau^{(\ell)}(\gamma_{\i}z)\eqdf\tau(\gamma_{i_{1}\ldots i_{\ell}}z)+\tau(\gamma_{i_{2}\ldots i_{\ell}}z)+\ldots+\tau(\gamma_{i_{\ell}}z).
\]
Notice that there exists $M>0$ independent of $\ell$ such that for
all $z,\i$ above we have $\vert\tau^{(\ell)}(\gamma_{\i}z)\vert\leq\ell M$.
We will then use the following basic bound, valid for all $z_{1},z_{2}\in\C$:
\[
\left|e^{z_{1}}-e^{z_{2}}\right|\leq\vert z_{1}-z_{2}\vert e^{\max\left\{ \Re(z_{1}),\Re(z_{2})\right\} },
\]
which yields 
\[
\left|e^{-s\tau^{(\ell)}(\gamma_{\i}z)}-e^{-s_{0}\tau^{(\ell)}(\gamma_{\i}z)}\right|\leq M\ell\vert s-s_{0}\vert e^{\ell M\max\left\{ \vert s_{0}\vert,\vert s\vert\right\} }.
\]
Using unitarity of the representation $\rho_{n}^{0}$ and triangle
inequality leads to the very crude bound 
\begin{align*}
 & \left\Vert \sum_{\i\rightarrow j,\,|\i|=\ell}\left(e^{-s\tau^{(\ell)}(\gamma_{\i}z)}-e^{-s_{0}\tau^{(\ell)}(\gamma_{\i}z)}\right)\rho_{n}^{0}(\gamma_{\i}^{-1})f(\gamma_{\i}z)\right\Vert _{V_{n}^{0}}\\
\leq & M\ell\vert s-s_{0}\vert e^{\ell M\max\left\{ \vert s_{0}\vert,\vert s\vert\right\} }\#\mathcal{W_{\ell}}\left(\sup_{\i\rightarrow j}\sup_{z\in D_{j}}\Vert f(\gamma_{\i}z)\Vert_{V_{n}^{0}}\right),
\end{align*}
which then gives
\[
\Vert\L_{s,\rho_{n}^{0}}^{\ell}(f)-\L_{s_{0},\rho_{n}^{0}}^{\ell}(f)\Vert^{2}\leq m(\D)\left(M\ell\vert s-s_{0}\vert\right)^{2}e^{2\ell M\max\left\{ \vert s_{0}\vert,\vert s\vert\right\} }\left(\#\mathcal{W_{\ell}}\right)^{2}\left(\sup_{\i\rightarrow j}\sup_{z\in D_{j}}\Vert f(\gamma_{\i}z)\Vert_{V_{n}^{0}}^{2}\right).
\]
By the reproducing property of Bergman's kernel, and since all $\gamma_{\i}$
map uniformly each $D_{j}$ in a compact subset of $\D$, we deduce
that there exists a constant $C>0$ uniform in $\ell$ such that 
\[
\sup_{\i\rightarrow j}\sup_{z\in D_{j}}\Vert f(\gamma_{\i}z)\Vert_{V_{n}^{0}}^{2}\leq C\Vert f\Vert_{\H(\D,V_{n}^{0})}^{2}.
\]
 The proof is now done if we use the fact that $s,s_{0}$ are in a
compact set.
\end{proof}
We then define the following operators $A_{i}(s):\H(\D,\C)\rightarrow\H(\D,\C)$
by
\begin{equation}
A_{i}(s)[f](x)\eqdf\mathbf{1}\{x\notin D_{i}\}e^{-s\tau(\gamma_{\bar{i}}(x))}f(\gamma_{\bar{i}}(x)).\label{eq:A_i-def}
\end{equation}
Being a composition operator which maps admissible discs $D_{i}$
compactly into $\D$, it is easy to see that each operator $A_{i}(s)$
is trace class on $\H(\D;\C).$ If $\i=(i_{1},i_{2},\ldots,i_{\ell})$
is an admissible word, we can then define the following composition
\[
A(s)_{\i}\eqdf A_{i_{1}}(s)A_{i_{2}}(s)\ldots A_{i_{\ell}}(s)[f]=\mathbf{1}\{x\notin D_{i_{1}}\}e^{-s\tau^{(\ell)}(\gamma_{\bar{\i}}x)}f\circ\gamma_{\bar{\i}}(x),
\]
where we have set for all admissible word $\j$ of length $\ell$,
$\tau^{(\ell)}(\gamma_{\j}z)=\tau(\gamma_{j_{1}\ldots j_{\ell}}z)+\tau(\gamma_{j_{2}\ldots j_{\ell}}z)+\ldots+\tau(\gamma_{j_{\ell}}z).$
It will be useful in the following to have a kernel formula for the
action of the adjoint operator $A(s)_{\i}^{*}$.
\begin{lem}
Using the above notations, we have for all $f\in\H(\D,\C)$, 
\[
A(s)_{\i}^{*}[f](x)=\int_{\D\setminus D_{i_{1}}}e^{-\overline{s}\overline{\tau^{(\ell)}(\gamma_{\overline{\i}}z)}}B_{\D}(x,\gamma_{\overline{\i}}z)f(z)dm(z).
\]
\end{lem}

\begin{proof}
Using the reproducing kernel property, we have for all $f,g\in\H(\D,\C)$,
\[
\langle A(s)_{\i}f,g\rangle=\int_{\D}\mathbf{1}\{z\notin D_{i_{1}}\}e^{-s\tau^{(\ell)}(\gamma_{\overline{\i}}z)}\int_{\D}B_{D}(\gamma_{\overline{\i}}z,w)f(w)dm(w)\overline{{g(z)}}dm(z).
\]
We observe that since there exists a compact set $K\subset\D$ such
that for all $z\notin D_{i_{1}}$, $\gamma_{\bar{\i}}(z)\in K,$ we
have (by using the explicit formula for the Bergman kernel)
\[
\sup_{z,w\in\D}\vert B_{D}(\gamma_{\bar{\i}}z,w)\mathbf{1}\{z\notin D_{i_{1}}\}\vert<\infty.
\]
Therefore we have
\[
\int_{\D}\mathbf{\int_{\D}1}\{z\notin D_{i_{1}}\}\left|e^{-s\tau^{(\ell)}(\gamma_{\overline{\i}}z)}\right|\vert B(\gamma_{\overline{\i}}z,w)f(w)g(z)\vert dm(z)dm(w)<\infty,
\]
and we can use Fubini's theorem to write
\[
\langle A(s)_{\i}f,g\rangle=\int_{\D}f(w)\overline{\int_{\D}\mathbf{1}\{z\notin D_{i_{1}}\}\overline{e^{-s\tau^{(\ell)}(\gamma_{_{\bar{\i}}}z)}}B_{\D}(w,\gamma_{\overline{\i}}z)g(z)dm(z)}dm(w),
\]
where we have used the fact that $B_{\D}(z,w)=\overline{B_{\D}(w,z)}$
and the proof is done.
\end{proof}
Notice that we have $A_{i}^{*}(s)\neq A_{\overline{i}}(s),$ which
is a source of non-selfadjointness issues and (besides the obvious
infinite dimensional setting) one of the main differences with \cite{BordenaveCollins},
where a symmetry hypothesis is assumed. See $\S\S$\ref{subsec:Conjugation-between-transfer}
for a more detailed discussion of these differences to \cite{BordenaveCollins}.

In this paper, we will have to deal with integral operators acting
on $\H(\D,\C)$ (or their vector valued extensions). We need a simple
criterion to compute their traces, which is as follows.
\begin{lem}
\label{lem:trace-formula}Let $T_{\mathcal{\mathscr{K}}}:\H(\D,\C)\rightarrow\H(\D,\C)$
be an integral operator whose kernel $\mathscr{K}(z,w)$ satisfies
the following properties: assume that $\mathscr{K\in}C^{0}(\overline{\D\times\D}),$
and that for all $w\in\D$, the map $z\mapsto\mathscr{K}(z,w)$ is
holomorphic and extends to a $C^{2}$ function on the boundary $\partial\D$.
Then, if $T_{\mathscr{K}}$ is trace class, we have the identity
\[
\tr(T_{\mathscr{K}})=\int_{\D}\mathscr{K}(z,z)dm(z).
\]
\end{lem}

\begin{proof}
It is enough to work in the unit disc $\mathbb{D}$ to simplify. Let
$(e_{p})_{\{p\in\N\}}$ be the Hilbert basis of $\H(\mathbb{D},\C)$
given explicitly by 
\[
e_{p}(z)=\sqrt{\frac{p+1}{\pi}}z^{p}.
\]
Because $T_{\mathscr{K}}$ is assumed to be trace class, we have
\[
\tr(T_{\mathscr{K}})=\sum_{p=0}^{\infty}\langle T_{\mathscr{K}}e_{p},e_{p}\rangle=\sum_{p}\int_{\mathbb{D}}\int_{\mathbb{D}}\mathscr{K}(z,w)e_{p}(w)dm(w)\overline{e_{p}(z)}dm(z),
\]
which by Fubini can be rewritten as
\[
\tr(T_{\mathscr{K}})=\sum_{p}\int_{\mathbb{D}}\int_{\mathbb{D}}\mathscr{K}(z,w)\overline{e_{p}(z)}dm(z)e_{p}(w)dm(w).
\]
Since $z\mapsto\mathscr{K}(z,w)$ is holomorphic on the unit disc,
we know that we have the convergent Taylor expansion
\[
\mathscr{K}(z,w)=\sum_{p=0}^{\infty}a_{\ell}(w)z^{\ell},
\]
where for all $0<r<1$, we can write
\[
a_{\ell}(w)=\frac{1}{2\pi}\int_{0}^{2\pi}\mathscr{K}(re^{i\theta},w)\left(re^{i\theta}\right)^{-\ell}d\theta.
\]
 because $z\mapsto\mathscr{K}(z,w)$ is actually continuous up to
the boundary we can write
\[
a_{\ell}(w)=\frac{1}{2\pi}\int_{0}^{2\pi}\mathscr{K}(e^{i\theta},w)\left(e^{i\theta}\right)^{-\ell}d\theta,
\]
and use the $C^{2}$ regularity to integrate by parts two times which
yields uniformly for all $w\in\overline{\mathbb{D}}$,
\[
\vert a_{\ell}(w)\vert=O\left(\ell^{-2}\right).
\]
 We now notice that 
\[
\int_{\mathbb{D}}\mathscr{K}(z,w)\overline{e_{p}(z)}dm(z)=a_{p}(w)\sqrt{\frac{\pi}{p+1}},
\]
which shows that uniformly for all $w\in\overline{\mathbb{D}}$,
\[
\left|\int_{\mathbb{D}}\mathscr{K}(z,w)\overline{e_{p}(z)}dm(z)\right|\vert e_{p}(w)\vert=O\left(p^{-2}\right),
\]
and the series 
\[
\sum_{p}\int_{\mathbb{D}}\mathscr{K}(z,w)\overline{e_{p}(z)}dm(z)e_{p}(w)
\]
are therefore uniformly convergent (with limit $\mathscr{K}(w,w)$)
on $\overline{\mathbb{D}},$ which allows to write
\[
\tr(T_{\mathscr{K}})=\int_{\mathbb{D}}\left(\sum_{p=0}^{\infty}\int_{\mathbb{D}}\mathscr{K}(z,w)\overline{e_{p}(z)}dm(z)e_{p}(w)\right)dm(w)=\int_{\D}\mathscr{K}(w,w)dm(w),
\]
and the proof is done.
\end{proof}
\begin{lem}
\label{lem:Trace-of-product-bound}Let $\K$ be any compact subset
of $\C$. There is $C=C(\K)>0$ such that for any $s=r+it\in\K$,
all $m>0$ and all 2m-tuples of admissible sequences $\i^{1},\ldots,\i^{m},\mathbf{j}^{1},\ldots,\mathbf{j}^{m}$
, we have
\[
\left|\mathrm{tr}\left(A(s)_{\i^{1}}A(s)_{\j^{1}}^{*}A(s)_{\i^{2}}A(s)_{\j^{2}}^{*}\cdots A(s)_{\i^{m}}A(s)_{\j^{m}}^{*}\right)\right|\leq C^{m}\prod_{k=1}^{m}\Upsilon_{\i^{k}}^{r}\Upsilon_{\j^{k}}^{r}.
\]
\end{lem}

\begin{proof}
Given an admissible word $\i$, we denote by $\mathscr{K_{\i}}(z,w)$
the integral kernel of $A(s)_{\i}$ which is given by
\[
\mathscr{K_{\i}}(z,w)=\mathbf{1}\{z\notin D_{i_{1}}\}e^{-s\tau^{(\vert\i\vert)}(\gamma_{\overline{\i}}z)}B_{\D}(\gamma_{\overline{\i}}z,w).
\]
Similarly we will denote by $\mathscr{K_{\i}^{*}}(z,w)$ the integral
kernel of $A(s)_{\i}^{*},$ which is
\[
\mathscr{K_{\i}^{*}}(z,w)=\mathbf{1}\{w\notin D_{i_{1}}\}\overline{e^{-s\tau^{(\vert\i\vert)}(\gamma_{\overline{\i}}w)}}B_{\D}(z,\gamma_{\overline{\i}}w).
\]
An important observation is due to the fact that each $\gamma_{\bar{\i}}$
maps $\D\setminus D_{i_{1}}$ uniformly and compactly into $\D$,
one can use the explicit formula for the Bergman kernel to prove that
there exists $C_{0}>0$ such that for all admissible $\i$, 
\[
\sup_{z,w\in\D}\left|\mathbf{1}\{z\notin D_{i_{1}}\}B(\gamma_{\overline{\i}}z,w)\right|\leq C_{0}\ \mathrm{and\ }\sup_{z,w\in\D}\left|\mathbf{1}\{w\notin D_{i_{1}}\}B(z,\gamma_{\overline{\i}}w)\right|\leq C_{0}.
\]
Because each kernel $\mathscr{K_{\i}}(z,w),\ \mathscr{K_{\j}^{*}}(z,w)$
is bounded on $\D\times\D$ (by the above observation), we can use
Fubini's theorem to compute the kernel of
\[
A(s)_{\i^{1}}A(s)_{\j^{1}}^{*}A(s)_{\i^{2}}A(s)_{\j^{2}}^{*}\cdots A(s)_{\i^{m}}A(s)_{\j^{m}}^{*},
\]
which is given by
\[
\mathscr{\mathscr{K}}_{\i^{1},\j^{1},\ldots,\i^{m},\j^{m}}(z,w)=
\]
\[
\int_{\D^{2m-1}}\mathscr{\mathscr{K}}_{\i^{1}}(z,w_{1})\mathscr{\mathscr{K}}_{\j^{1}}^{*}(w_{1,}w_{2})\mathscr{\mathscr{K}}_{\i^{2}}(w_{2},w_{3})\ldots\mathscr{\mathscr{K}^{*}}_{\j^{m}}(w_{2m-1,}w)dm(w_{1})\ldots dm(w_{2m-1}).
\]
 We can then check that $z\mapsto\mathscr{\mathscr{K}}_{\i^{1},\j^{1},\ldots,\i^{m},\j^{m}}(z,w)$
is holomorphic with an analytic extension to $\partial\D$. Being
a product of trace class operators, $A(s)_{\i^{1}}A(s)_{\j^{1}}^{*}A(s)_{\i^{2}}A(s)_{\j^{2}}^{*}\cdots A(s)_{\i^{m}}A(s)_{\j^{m}}^{*}$
is therefore trace class and we can use Lemma \ref{lem:trace-formula}
to obtain (again rearranging the order of integration by Fubini)
\[
\mathrm{tr}\left(A(s)_{\i^{1}}A(s)_{\j^{1}}^{*}A(s)_{\i^{2}}A(s)_{\j^{2}}^{*}\cdots A(s)_{\i^{m}}A(s)_{\j^{m}}^{*}\right)
\]
\[
=\int_{\D^{2m}}\mathscr{K}_{\i^{1}}(w_{2m},w_{1})\mathscr{\mathscr{K}}_{\j^{1}}^{*}(w_{1,}w_{2})\ldots\mathscr{\mathscr{K}}_{\i^{m}}(w_{2m-2},w_{2m-1})\mathscr{\mathscr{K}}_{\j^{m}}^{*}(w_{2m-1,}w_{2m})dm(w_{1})\ldots dm(w_{2m}).
\]
 We now recall that for all $z\in\D\setminus D_{i_{1}}$, 
\[
e^{-s\tau^{(\vert\i\vert)}(\gamma_{\overline{\i}}z)}=(\gamma_{\overline{\i}}'(z))^{s}=e^{s\log\gamma'_{\overline{\i}}(z)},
\]
where $\log(w)=\log\vert w\vert+i\arg(w)$ is the principal branch
of the complex logarithm on $\C\setminus(-\infty,0]$. Using (\ref{eq:upsilon-deriv-eq}),
this leads to the bound for $s\in\K$
\[
\left|(\gamma_{\overline{\i}}'(z))^{s}\right|\leq e^{\Re(s\log\gamma_{\overline{\i}}'(z))}\leq C_{1}(\mathcal{K})\Upsilon_{\overline{\i}}^{\Re(s)},
\]
where $C_{1}(\mathcal{K})>0$. Using the mirror estimate (\ref{eq:upsilon-mirror-eq}),
we end up with (up to a slight change of constant)
\[
\left|(\gamma_{\overline{\i}}'(z))^{s}\right|\leq C_{2}(\mathcal{K})\Upsilon_{\i}^{\Re(s)}.
\]
We have therefore if $\Re(s)=r$
\[
\left|\mathrm{tr}\left(A(s)_{\i^{1}}A(s)_{\j^{1}}^{*}A(s)_{\i^{2}}A(s)_{\j^{2}}^{*}\cdots A(s)_{\i^{m}}A(s)_{\j^{m}}^{*}\right)\right|\leq C_{0}^{2m}\left(C_{2}(\mathcal{K})\right)^{2m}\mathrm{(Vol}(\D))^{2m}\prod_{k=1}^{m}\Upsilon_{\i^{k}}^{r}\Upsilon_{\j^{k}}^{r},
\]
which concludes the proof.
\end{proof}

\section{Random permutations and random graphs\label{sec:Random-Permutations-and-graphs}}

\subsection{Symmetric random permutations}

We write $[n]\eqdf\{1,\ldots,n\}$ and we write $S_{n}$ for the symmetric
group of permutations of $[n]$. We view $d\geq2$ as a fixed integer
as in $\S\S$\ref{subsec:Boundary-coding-of}. As before, for $i\in[2d]$
let $\bar{i}\eqdf i+d\bmod2d$ so $\bar{i}\in[2d]$. We say that a
tuple in $(S_{n})^{2d}$ is \emph{symmetric }if $\sigma_{\bar{i}}=\sigma_{i}^{-1}$.
To obtain symmetric random permutations we choose $\sigma_{1},\ldots,\sigma_{d}$
independently and uniformly in $S_{n}$. (Of course, choosing these
plays the same role as in $\S\S$\ref{subsec:Random-covering-spaces,}
and hence is the same thing as choosing our random cover $X_{\phi}$.)

We let $\sigma_{\bar{i}}=\sigma_{i}^{-1}$ to extend the indices of
$\sigma_{i}$ to $[2d]$ and we write $S_{i}$ for the matrix representing
$\sigma_{i}$ as a 0-1 matrix. We also let 
\[
\underline{S}_{i}\eqdf S_{i}-\frac{1}{n}\mathbf{1}\otimes\mathbf{1},
\]
this is just the matrix of $\sigma_{i}$ acting in $V_{n}^{0}$.

\subsection{Graphs and paths}

First following \cite[Def. 6]{BordenaveCollins} we define the category
of graphs that we work with.
\begin{defn}
A colored edge is an equivalence class of element $(x,i,y)$ of $[n]\times[d]\times[n]$
with respect to the equivalence relation generated by
\begin{equation}
(x,i,y)\sim(y,\bar{i},x).\label{eq:equiv-relation}
\end{equation}
A colored graph is a pair consisting of a vertex set $V\subset[n]$
and a set of colored edges $E$ as above. It is clear that any colored
graph has underlying multigraph obtained by forgetting colors. When
we refer to paths, cycles, etc, in colored graphs, they are just paths
and cycles in the underlying multigraph.

For symmetric permutations $(\sigma_{1},\ldots,\sigma_{2d})$ in $(S_{n})^{2d}$,
$G^{\sigma}$ is the colored graph with vertex set $[n]$ and colored
edges given by $[x,i,y]$ such that $\sigma_{i}(x)=y$, modulo the
equivalence relation on colored edges. 
\end{defn}

We also introduce a way to talk about paths and the graphs they trace
out.
\begin{defn}
\label{def:paths}The set $E$ is the collection of pairs $e=(x,i)$
with $x\in[n]$ and $i\in[2d]$. Each of these be thought of as a
half-edge emanating from $x$ and labeled by $i$. A \emph{path} of
length $k$ is a sequence in $E^{k}$. For $k\in\N$ let $\gamma=(\gamma_{1},\ldots,\gamma_{k})\in E^{k}$
be a path, and write each $\gamma_{t}=(x_{t},i_{t})\in[n]\times[2d]$.
\begin{itemize}
\item Let $V_{\gamma}=\{x_{t}\,:\,1\leq t\leq k\}$ and $E_{\gamma}=\{[x_{t},i_{t},x_{t+1}]\,:\,1\leq t\leq k-1\}$.
We denote by $G_{\gamma}$ the colored graph with vertex set $V_{\gamma}$
and colored edges $E_{\gamma}$ modulo the equivalence relation (\ref{eq:equiv-relation}).
\item A path is non-backtracking if for each $1\leq t\leq k-1$ as above,
$i_{t}\neq\bar{i}_{t+1}$. We write $\Gamma^{k}$ for the subset of
of non-backtracking paths in $E^{k}$. For $e,f\in E$ we write $\Gamma_{ef}^{k}\subset\Gamma^{k}$
for the non-backtracking paths with $\gamma_{1}=e$, $\gamma_{k}=f$.
\end{itemize}
\end{defn}

\begin{rem}
Note that in Definition \ref{def:paths} the paths, etc have nothing
a priori to do with $G^{\sigma}$. Rather, their role is as objects
that could \emph{potentially} appear in $G^{\sigma}$.
\end{rem}

\begin{defn}
\label{def:color-sequence}Given a non-backtracking path $\gamma=((x_{1},i_{1}),\ldots,(x_{k},i_{k}))\in\Gamma^{k}$,
the \emph{color sequence }of the path is 
\[
\i(\gamma)\eqdf(i_{1},\ldots,i_{k}).
\]
\end{defn}

\subsection{Tangles}

One important notion in the works of Bordenave-Collins \cite{BordenaveCollins},
originating in Friedman \cite{Friedman}, is \emph{tangles. }These
relate to events of small probability that make large, argument-destroying,
contributions to the expected value of traces. The following definition
is \cite[Def. 19]{BordenaveCollins}
\begin{defn}
\label{def:tangles}Let $H$ be a colored graph. $H$ is \emph{tangle-free}
if it contains at most one cycle. For any vertex $x$ of $H$, let
$(H,x)_{\ell}$ denote the subgraph of $H$ containing all colored
edges and vertices that belong to a path of length at most $\ell$
beginning at $x$. $H$ is $\ell$-tangle-free if for any vertex $x$,
$(H,x)_{\ell}$ is tangle-free, and $H$ is $\ell$-tangled otherwise.
We say that a path $\gamma\in E^{k}$ is tangle-free if $G_{\gamma}$
is tangle-free. We write $F^{k}$ (resp. $F_{ef}^{k}$) for the collection
of tangle-free non-backtracking paths in $\Gamma^{k}$ (resp. $\Gamma_{ef}^{k}$).
\end{defn}

The following lemma appearing in \cite[Lemma 23]{BordenaveCollins}
gives the probability that $G^{\sigma}$ is $\ell$-tangled when $(\sigma_{i})$
are symmetric random permutations in $(S_{n})^{2d}$.
\begin{lem}
\label{lem:tangle-probability}There is a constant $c>0$ such that
for $(\sigma_{i})$ symmetric random permutations in $(S_{n})^{2d}$
and $1\leq\ell\leq\sqrt{n}$, the probability that $G^{\sigma}$ is
$\ell$-tangled is at most
\[
c\frac{\ell^{3}(2d-1)^{4\ell}}{n}.
\]
\end{lem}

\section{Non-backtracking operator and path decomposition}

\subsection{Conjugation between transfer operator and non-backtracking operator\label{subsec:Conjugation-between-transfer}}

The purpose of this section is to show how the transfer operator,
after a unitary conjugation, is roughly of the form as the non-backtracking
operator of \cite{BordenaveCollins}. This conjugation is not essential
for the argument and merely makes it easier to read the current work
alongside \cite{BordenaveCollins}.

Recall from $\S\S$\ref{subsec:Transfer-operator-and} that our transfer
operator

\[
\L_{s,\rho_{n}}[f](x)\eqdf\sum_{i\in\A}e^{-s\tau(\gamma_{i}(x))}\rho_{n}(\gamma_{i})^{-1}f(\gamma_{i}(x))
\]
acts on $\H(\D,V_{n})$ and the restriction of this operator to $\H(\D,V_{n}^{0})$
coincides with $\L_{s,\rho_{n}^{0}}$; we aim to prove for certain
values of $s$ that $\L_{s,\rho_{n}^{0}}$ has no eigenvalue 1 (cf.
Proposition \ref{prop:new-resonances-and-evalue-1}).

Let
\begin{align*}
K_{0} & \eqdf\H(\D)\otimes V_{n}^{0}\subset\H(\D)\otimes\ell^{2}([n]).
\end{align*}
We let $S_{i}$ denote the 0-1 matrix of $\sigma_{i}=\phi(\gamma_{i})$
as in $\S\S$\ref{subsec:Random-covering-spaces,} (recall also that
$\rho_{n}$ and $\phi$ are coupled). For each element $e=(x,i)\in E$
we consider the subspace
\[
\H_{e}\eqdf\H(D_{\bar{i}})\otimes\C x\subset\H(\D)\otimes\ell^{2}([n]).
\]
(The reason $\bar{i}$ appears here is to make things line up with
Bordenave-Collins machinery). In $\ell^{2}([n])$, we write $y$ for
the indicator function of $y$ (this corresponds to identification
of $\ell^{2}([n])$ with formal complex linear combinations of $[n]$).
This gives 
\[
\H(\D)\otimes\ell^{2}([n])=\bigoplus_{e\in E}\H_{e}.
\]
Write $\p_{e}$ for the orthogonal projection onto $\H_{e}$. For
any endomorphism $M$ of $\H(\D)\otimes\ell^{2}([n])$ and for each
$e,f\in E$ we write 
\[
M_{ef}\eqdf\p_{e}M\p_{f}.
\]
 Clearly $M$ is determined by the values $M_{ef}$. Let $\EE_{xy}:\ell^{2}([n])\to\ell^{2}([n])$
be the map with
\[
(\EE_{xy})(y')=\delta_{yy'}x
\]
and $\delta_{yy'}$ is the Kronecker delta. 

Recall the definition of operators $A_{j}(s)$ from (\ref{eq:A_i-def}).
For $e=(x,i)$ and $f=(y,j)$ we define
\begin{equation}
B(s)_{ef}\eqdf\mathbf{1}\{i\neq\bar{j}\}(S_{i})_{xy}\p_{e}[A_{j}(s)\otimes\mathcal{E}_{xy}]\p_{f};\label{eq:B-formula-exact}
\end{equation}
this defines $B(s)\in\End\left(\H(\D)\otimes\ell^{2}([n])\right)$.
The operator $B(s)$ is in a similar form to the operator $B$ defined
by Bordenave and Collins in \cite[(12)]{BordenaveCollins} (cf. also
the display equation following \cite[(12)]{BordenaveCollins}). However
our $B(s)$ and Bordenave-Collins' $B$ differ in the following ways: 
\begin{itemize}
\item our $\H(\D)$ is the analog of $\C^{r}\otimes\ell^{2}([2d])$, however,
$\H(\D)$ does not split as a tensor product in this way. 
\item Also, our $A_{j}(s)$ is slightly more general (in the Bordenave-Collins
setting, it is equivalent to considering $a_{j}\in\End(\C^{r}\otimes\ell^{2}([2d]))$
rather than $a_{j}\in\End(\C^{r})$).
\item each $A_{j}(s)$ acts on a infinite dimensional Hilbert space $\H(\D)$.
\item \textbf{most critically}, the symmetry condition $A_{i}=A_{\bar{i}}^{*}$
that Bordenave-Collins assume does not hold here.
\end{itemize}
These differences, and \emph{especially the lack of the symmetry condition},
have to be worked around in the sequel.
\begin{prop}
\label{prop:conjugation}There is a unitary operator $U:\H(\D,V_{n})\to\H(\D)\otimes\ell^{2}([n])$
such that
\begin{enumerate}
\item $U$ conjugates the holomorphic family of trace class operators $\L_{s,\rho_{n}}$
to the family of operators $B(s)$.
\item $U$ conjugates the holomorphic family of trace class operators $\L_{s,\rho_{n}^{0}}$
to the family of operators $B(s)\lvert_{K_{0}}$.
\end{enumerate}
As such, $B(s)$ and $B(s)\lvert_{K_{0}}$ are holomorphic families
of trace class operators and $\|\L_{s,\rho_{n}^{0}}^{\ell}\|=\|B(s)^{\ell}\lvert_{K_{0}}\|$
for all $s\in\C$ and $\ell\in\N.$
\end{prop}

\begin{proof}
First of all, by the obvious identification
\[
\H(\D,V_{n})\cong\H(\D)\otimes\ell^{2}([n])
\]
the operator $\L_{s,\rho_{n}}$ is conjugated to $b(s)\in\End(\H(\D)\otimes\ell^{2}([n]))$,
\[
b(s)\eqdf\sum_{i}A_{i}(s)\otimes\sigma_{i}.
\]
Consider the unitary operator $Q\in\End(\H(\D)\otimes\ell^{2}([n]))$
defined for $e=(x,i)$ and $f=(y,j)$ by 
\[
Q_{ef}\eqdf\p_{e}\left(\mathrm{Id}_{\H(D_{\bar{j}})}\otimes\sigma_{j}^{-1}\right)\p_{f}.
\]
Then for $F\otimes y\in\H_{f}$, $f=(y,j)$, $F\in\H(D_{\bar{j}})$,
using $A_{i}(s)F=0$ unless $i=j$ we obtain
\begin{align*}
[Q^{-1}b(s)Q]_{ef}[F\otimes y] & =\p_{e}Qb(s)[F\otimes\sigma_{j}^{-1}(y)]\\
 & =\p_{e}Q[\left(A_{j}(s)F\right)\otimes y]\\
 & =\p_{e}\left(A_{j}(s)F\otimes\sigma_{i}^{-1}(y)\right)\\
 & =\mathbf{1}\{i\neq\bar{j}\}\mathbf{1}\{\sigma_{i}(x)=y\}\p_{e}\left(A_{j}(s)F\otimes x\right)\\
 & =B(s)_{ef}[F\otimes y]
\end{align*}
proving that $Q^{-1}b(s)Q=B(s)$. It is also not hard to see that
$B(s)$ leaves $K_{0}$ invariant and $Q$ also conjugates the restriction
$B(s)\lvert_{K_{0}}$ to $\L_{s,\rho_{n}^{0}}$. Finally, noting that
the conjugation between $\L_{s,\rho_{n}}$ and $B(s)$ does not depend
on $s$ completes the proof.
\end{proof}

\subsection{Bordenave's path decomposition}

Here we perform a path decomposition method that originates in work
of Bordenave \cite{bordenave2015new}. Recall from Definition \ref{def:tangles}
that $F^{\ell+1}$ is the collection of tangle-free paths in $\Gamma^{\ell+1}$.
For $1\leq k\leq\ell$ we write $F_{k}^{\ell+1}$ for the collection
of paths $\gamma$ in $\Gamma^{\ell+1}$ such that the first $k$
half-edges of $\gamma$ (in $\Gamma^{k})$ form a tangle-free path
and the last $\ell-k+1$ half-edges $($in $\Gamma^{\ell-k+1})$ also
form a tangle-free path.

Recall the definition of $\i(\gamma)$ from Definition \ref{def:color-sequence}
and the notation $\hat{\i}=(i_{2},\ldots,i_{n})$ from $\S\S$\ref{subsec:Boundary-coding-of}.
As in \cite[pg. 834]{BordenaveCollins} we obtain with $e=(x,i)$
and $f=(y,j)$
\begin{equation}
B(s)_{ef}^{\ell}=\p_{e}\left(\sum_{\gamma\in\Gamma_{ef}^{\ell+1}}(A(s)_{\hat{\i}(\gamma)}\otimes\EE_{xy})\prod_{t=1}^{\ell}(S_{i_{t}})_{x_{t}x_{t+1}}\right)\p_{f}\label{eq:Bl-expression}
\end{equation}

For $\ell\in\N$ we define as in \cite[(32), (34)]{BordenaveCollins}
\begin{align*}
\uB(s)^{(\ell)} & \eqdf\sum_{\gamma=((x_{1},i_{1}),\ldots,(x_{\ell+1},i_{\ell+1}))\in F^{\ell+1}}\p_{(x_{1},i_{1})}[A(s)_{\hat{\i}(\gamma)}\otimes\mathcal{E}_{x_{1}x_{\ell+1}}]\p_{(x_{\ell+1},i_{\ell+1})}\left(\prod_{t=1}^{\ell}(\underline{S}_{i_{t}})_{x_{t}x_{t+1}}\right),\\
R_{k}(s)^{(\ell)} & \eqdf\sum_{\gamma=((x_{1},i_{1}),\ldots,(x_{\ell+1},i_{\ell+1}))\in F_{k}^{\ell+1}\backslash F^{\ell+1}}\\
 & \,\,\,\,\,\,\,\,\p_{(x_{1},i_{1})}[A(s)_{\hat{\i}(\gamma)}\otimes\mathcal{E}_{x_{1}x_{\ell+1}}]\p_{(x_{\ell+1},i_{\ell+1})}\left(\prod_{t=1}^{k-1}(\underline{S}_{i_{t}})_{x_{t}x_{t+1}}\right)\left(\prod_{t=k+1}^{\ell}(S_{i_{t}})_{x_{t}x_{t+1}}\right).
\end{align*}
 
\begin{rem}
\label{rem:About-trace-class}Since each $A(s)_{i}$ is trace class
(see $\S\S$\ref{subsec:Deterministic-a-priori-bounds}), both $\uB(s)^{(\ell)}$
and $R_{k}(s)^{(\ell)}$ are trace class operators for any $\ell\in\N$.
\end{rem}

The following is proved by Bordenave and Collins in the course of
proving \cite[Lemma 20]{BordenaveCollins} (see the last display equation
before \cite[Lemma 20]{BordenaveCollins}).
\begin{lem}
\label{lem:Bell-to-B(ell)-and-Rkl}Let $\ell\geq1$ be an integer
and suppose $G^{\sigma}$ is $\ell$-tangle free. Then for all $F\in K_{0}$,
\begin{equation}
B(s)^{\ell}F=\uB(s)^{(\ell)}-\frac{1}{n}R_{k}(s)^{(\ell)}F.\label{eq:b-bell-rk}
\end{equation}
Hence
\[
\|B(s)^{\ell}\lvert_{K_{0}}\|\leq\|\uB(s)^{(\ell)}\|+\frac{1}{n}\sum_{k=1}^{\ell}\|R_{k}(s)^{(\ell)}\|.
\]
\end{lem}

\begin{proof}
It clearly suffices to prove the first statement. This is similar
to \cite[pgs. 836-837]{BordenaveCollins} but there are subtle differences
so we give the details for completeness. Let $e=(x,i)$ and $f=(y,j)$
in $E$. Beginning with the expression (\ref{eq:Bl-expression}) for
$B(s)_{ef}^{\ell}$, \emph{on the event that $G^{\sigma}$ is} $\ell$-\emph{tangle
free} we can write 
\begin{align}
B(s)_{ef}^{\ell}= & \p_{e}\left(\sum_{\gamma\in\Gamma_{ef}^{\ell+1}}(A(s)_{\hat{\i}(\gamma)}\otimes\EE_{xy})\prod_{t=1}^{\ell}(S_{i_{t}})_{x_{t}x_{t+1}}\right)\p_{f}\nonumber \\
= & \p_{e}\left(\sum_{\gamma\in F_{ef}^{\ell+1}}(A(s)_{\hat{\i}(\gamma)}\otimes\EE_{xy})\prod_{t=1}^{\ell}(\underline{S}_{i_{t}})_{x_{t}x_{t+1}}\right)\p_{f}\label{eq:bbarline}\\
 & +\frac{1}{n}\sum_{k=1}^{\ell}\p_{e}\left(\sum_{\gamma\in F_{ef}^{\ell+1}}(A(s)_{\hat{\i}(\gamma)}\otimes\EE_{xy})\prod_{t=1}^{k-1}(\underline{S}_{i_{t}})_{x_{t}x_{t+1}}\prod_{t=k+1}^{\ell}(S_{i_{t}})_{x_{t}x_{t+1}}\right)\p_{f},\label{eq:rkline}
\end{align}
where the last equality used 
\[
\prod_{t=1}^{\ell}(S_{i_{t}})_{x_{t}x_{t+1}}=\prod_{t=1}^{\ell}(\underline{S}_{i_{t}})_{x_{t}x_{t+1}}+\frac{1}{n}\sum_{k=1}^{\ell}\prod_{t=1}^{k-1}(\underline{S}_{i_{t}})_{x_{t}x_{t+1}}\prod_{t=k+1}^{\ell}(S_{i_{t}})_{x_{t}x_{t+1}}.
\]
 The operator on line (\ref{eq:bbarline}) is $\uB(s)_{ef}^{(\ell)}$
so it suffices to describe the operator on line (\ref{eq:rkline}).
We write 
\begin{align}
 & \p_{e}\left(\sum_{\gamma\in F_{ef}^{\ell+1}}(A(s)_{\hat{\i}(\gamma)}\otimes\EE_{xy})\prod_{t=1}^{k-1}(\underline{S}_{i_{t}})_{x_{t}x_{t+1}}\prod_{t=k+1}^{\ell}(S_{i_{t}})_{x_{t}x_{t+1}}\right)\p_{f}\nonumber \\
= & \p_{e}\left(\sum_{\gamma\in F_{k,ef}^{\ell+1}}(A(s)_{\hat{\i}(\gamma)}\otimes\EE_{xy})\prod_{t=1}^{k-1}(\underline{S}_{i_{t}})_{x_{t}x_{t+1}}\prod_{t=k+1}^{\ell}(S_{i_{t}})_{x_{t}x_{t+1}}\right)\p_{f}\label{eq:Me}\\
 & -\p_{e}\left(\sum_{\gamma\in F_{k,ef}^{\ell+1}\backslash F_{ef}^{\ell+1}}(A(s)_{\hat{\i}(\gamma)}\otimes\EE_{xy})\prod_{t=1}^{k-1}(\underline{S}_{i_{t}})_{x_{t}x_{t+1}}\prod_{t=k+1}^{\ell}(S_{i_{t}})_{x_{t}x_{t+1}}\right)\p_{e}.\label{eq:R_k-contribs}
\end{align}

The terms in line (\ref{eq:R_k-contribs}) give the required contributions
to $-\frac{1}{n}R_{k}(s)^{(\ell)}$ in (\ref{eq:b-bell-rk}) so it
suffices to show that the operator $M$ defined by 
\[
M_{ef}\eqdf\eqref{eq:Me}
\]
kills every element of $K_{0}$.

Let $J:\H(\D)\otimes\ell^{2}([n])\to\H(\D)\otimes\ell^{2}([n])$ denote
the operator defined for $e=(x,i),\,f=(y,j)$
\[
J_{ef}=\mathbf{1}\{i\neq\bar{j}\}\p_{e}[A_{j}(s)\otimes\mathcal{E}_{xy}]\p_{f}
\]
(contrast this to (\ref{eq:B-formula-exact})). Note that if $F=F_{0}\otimes v\in K_{0}$,
with $F_{0}\in\H(D_{\bar{j}})$ and $v=(v_{1},\ldots,v_{n})\in V_{n}^{0}$,
$\sum_{y\in[n]}v_{y}=0,$ then
\begin{align*}
J[F_{0}\otimes v] & =\sum_{\substack{e=(x,i)\\
i\neq\bar{j}
}
}\sum_{y\in[n]}\p_{e}[A_{j}(s)\otimes\mathcal{E}_{xy}]\p_{(y,j)}\left(F_{0}\otimes v\right)\\
 & =\sum_{\substack{e=(x,i)\\
i\neq\bar{j}
}
}\sum_{y\in[n]}v_{y}\p_{e}[A_{j}(s)\otimes\mathcal{E}_{xy}]\left(F_{0}\otimes y\right)\\
 & =\sum_{\substack{e=(x,i)\\
i\neq\bar{j}
}
}[\left(A_{j}(s)F_{0}\right)\lvert_{D_{\bar{i}}}\otimes x]\sum_{y\in[n]}v_{y}=0.
\end{align*}
This implies $K_{0}\subset\ker J$.

Since we still assume $G^{\sigma}$ is $\ell$-tangle-free, and hence
$k$-tangle free for all $k\in[\ell]$, we have
\begin{align*}
M_{ef} & =\p_{e}\left(\sum_{\gamma\in F_{k,ef}^{\ell+1}}(A(s)_{\hat{\i}(\gamma)}\otimes\EE_{xy})\prod_{t=1}^{k-1}(\underline{S}_{i_{t}})_{x_{t}x_{t+1}}\prod_{t=k+1}^{\ell}(S_{i_{t}})_{x_{t}x_{t+1}}\right)\p_{f}\\
 & =(\tilde{M}JB(s)^{\ell-k})_{ef}
\end{align*}
 for some operator $\tilde{M}$. Since $B(s)$ preserves $K_{0}$,
this implies $K_{0}$ is in the kernel of $M$ as required.
\end{proof}
Lemma \ref{lem:Bell-to-B(ell)-and-Rkl} tells us that bounds on both
$\|\uB(s)^{(\ell)}\|$ and $\|R_{k}(s)^{(\ell)}\|$ that hold with
high probability can be coupled with the fact that $G^{\sigma}$ is
$\ell$-tangle-free with high-probability (for $\ell$ in a suitable
range) to establish bounds for $\|B(s)^{\ell}\lvert_{K_{0}}\|$ that
hold with high probability. A high probability bound for $\|\uB(s)^{(\ell)}\|$
is the subject of $\S\S$\ref{sec:Norm-of-B(ell)} and high probability
bounds for $\|R_{k}(s)^{(\ell)}\|$ are the subject of $\S\S$\ref{sec:Norm-of-Rk(ell)}.

\section{High trace method}

\subsection{Norm of $\protect\uB(s)^{(\ell)}$\label{sec:Norm-of-B(ell)}}

The main result of this $\S$\ref{sec:Norm-of-B(ell)} is the following.
\begin{prop}
\label{prop:B(ell)-norm-bound}Let $\K$ be a compact subset of $\{s\,:\,\Re(s)>\frac{\delta}{2}\,\}.$
There is $c=c(\K)>1$, $\rho=\rho(\K)<1$ such that for all $s\in\K$,
for all $1\leq\ell\leq\log n$ the event
\[
\|\uB(s)^{(\ell)}\|\leq(\log n)^{20}\rho^{\ell}
\]
 holds with probability at least $1-c\exp\left(-\frac{\ell\log n}{c\log\log n}\right)$.
\end{prop}

Here we follow the structure of \cite[\S\S\S4.4.1]{BordenaveCollins}
up to Proposition \ref{prop:sum-over-iso-class}, when we make our
departure. The method used here to bound $\|\uB(s)^{(\ell)}\|$ is
the `high-trace' method originating in work of Füredi and Komlós \cite{FK}.

For $\ell,m\in\N$ we let $W_{\ell,m}$ denote the set of 
\[
\gamma=(\gamma_{1},\ldots,\gamma_{2m})\in(F^{\ell+1})^{2m}
\]
 such that for $j=1,\ldots,m$, the first half-edge (i.e. $\gamma_{2j,1}=(x_{2j,1},i_{2j,1})$)
of $\gamma_{2j}$ is the same as the first half-edge of $\gamma_{2j+1}$,
and the last half-edge of $\gamma_{2j-1}$ is the same as the last
half-edge of $\gamma_{2j}$. Here we let $\gamma_{2m+1}\eqdf\gamma_{1}$.

Following \cite[(36)]{BordenaveCollins} we get
\begin{align*}
\|\uB(s)^{(\ell)}\|^{2m}\leq & \tr\left(\left(\uB(s)^{(\ell)}(\uB(s)^{(\ell)})^{*}\right)^{m}\right)\\
= & \sum_{\substack{\gamma=(\gamma_{1},\ldots,\gamma_{2m})\in W_{\ell,m}\\
\gamma_{j}=((x_{j,1},i_{j,1}),\ldots,(x_{j,\ell+1},i_{j,\ell+1}))
}
}\left(\tr\left[A(s)_{\hat{\i}(\gamma_{1})}A(s)_{\hat{\i}(\gamma_{2})}^{*}\cdots A(s)_{\hat{\i}(\gamma_{2m-1})}A(s)_{\hat{\i}(\gamma_{2m})}^{*}\right]\right)\\
 & \cdot\prod_{j=1}^{2m}\prod_{t=1}^{\ell}(\underline{S}_{i_{j,t}})_{x_{j,t}x_{j,t+1}}.
\end{align*}
Here we used that all $A(s)_{\hat{\i}}$ and indeed also $\uB(s)^{(\ell)}$
are trace class operators (cf. Remark \ref{rem:About-trace-class}).

Therefore
\begin{align}
\E[\|\uB(s)^{(\ell)}\|^{2m}] & \leq\sum_{\substack{\gamma=(\gamma_{1},\ldots,\gamma_{2m})\in W_{\ell,m}\\
\gamma_{j}=((x_{j,1},i_{j,1}),\ldots,(x_{j,\ell+1},i_{j,\ell+1}))
}
}\left|\tr\left[A(s)_{\hat{\i}(\gamma_{1})}A(s)_{\hat{\i}(\gamma_{2})}^{*}\cdots A(s)_{\hat{\i}(\gamma_{2m-1})}A(s)_{\hat{\i}(\gamma_{2m})}^{*}\right]\right|\nonumber \\
 & \cdot\left|\E\left[\prod_{j=1}^{2m}\prod_{t=1}^{\ell}(\underline{S}_{i_{j,t}})_{x_{j,t}x_{j,t+1}}\right]\right|.\label{eq:exp-bl-bound-1}
\end{align}
Both the expectation and the absolute value of the trace in the right
hand side of the above inequality will be estimated in terms of topological
properties of $\gamma$, among other things. Recall that in Definition
\ref{def:paths} we defined for a path $\gamma\in E^{k}$ a colored
graph $G_{\gamma}$ that is roughly speaking the image of the path
$\gamma$. Now, for $\gamma=(\gamma_{1},\ldots,\gamma_{2m})\in W_{\ell,m}$
we define $G_{\gamma}$ to be the union of the graphs $G_{\gamma_{i}}$.
In other words, the vertices of $G_{\gamma}$ are the elements of
$[n]$ visited by $\gamma$ and the edges are just the ones followed
by $\gamma$ with their corresponding colors. We say that a colored
edge of $G_{\gamma}$ has \emph{multiplicity one }if it is traversed
exactly once by $\gamma$ (in either direction, so an edge traversed
in both directions is \textbf{not }multiplicity one).

Notice that for $\gamma\in W_{\ell,m}$, $G_{\gamma}$ is always connected.
We let $W_{\ell,m}(v,e)$ denote the collection of $\gamma\in W_{\ell,m}$
such that $G_{\gamma}$ has $v$ vertices and $e$ edges; by the previous
remark $W_{\ell,m}(v,e)$ is empty unless $e-v+1\geq0$.

We say two elements $\gamma$ and $\gamma'$ of $W_{\ell,m}$ are
\emph{isomorphic }if $\gamma'$ is obtained from $\gamma$ by changing
the labels (in $[n]$) of the vertices of $G_{\gamma}$ and changing
the colors (in $[2d])$ assigned to edges in $G_{\gamma}$; these
changes to $G_{\gamma}$ clearly induce a well-defined change of $\gamma$.
The reader can see \cite[pp. 841-842]{BordenaveCollins} for a more
formal definition.

Bordenave-Collins prove, by associating to each isomorphism class
in $W_{\ell,m}$ a canonical element, and cleverly counting these
canonial elements, the following result \cite[Lemma 25]{BordenaveCollins}.
\begin{lem}
\label{lem:iso-class-count-B}The number of isomorphism classes in
$W_{\ell,m}(v,e)$ is
\[
\leq\left(4d\ell m\right)^{6m\cdot\rank+10m}
\]
where\footnote{We choose to use $\rank$ rather than what Bordenave-Collins call
$\chi$, because of potential confusion with Euler characteristic.} $\rank\eqdf e-v+1$.
\end{lem}

\begin{rem}
It is in this lemma that the tangle-free hypothesis is crucially used.
\end{rem}

Bordenave-Collins also prove the following random matrix result in
\cite[Lemma 27]{BordenaveCollins}.
\begin{lem}
\label{lem:BC-random-matrix-1}There is a constant $c>0$ such that
if $2\ell m\leq\sqrt{n}$ and $\gamma=(\gamma_{1},\ldots,\gamma_{2m})\in W_{\ell,m}(v,e),$with
$\gamma_{j}=((x_{j,1},i_{j,1}),\ldots,(x_{j,\ell+1},i_{j,\ell+1}))$
for $j\in[2m]$, with $G_{\gamma}$ having $e_{1}$ edges of multiplicity
1, then
\[
\left|\E\left[\prod_{j=1}^{2m}\prod_{t=1}^{\ell}(\underline{S}_{i_{j,t}})_{x_{j,t}x_{j,t+1}}\right]\right|\leq c^{m+\rank}\left(\frac{1}{n}\right)^{e}\left(\frac{6\ell m}{\sqrt{n}}\right)^{\max\left(e_{1}-4\cdot\rank-4m,0\right)},
\]
where $\rank=e-v+1$.
\end{lem}

\begin{rem}
The quantity $\left|\E\left[\prod_{j=1}^{2m}\prod_{t=1}^{\ell}(\underline{S}_{i_{j,t}})_{x_{j,t}x_{j,t+1}}\right]\right|$
is constant on isomorphism classes in $W_{\ell,m}$.
\end{rem}

Now grouping terms in (\ref{eq:exp-bl-bound-1}) into isomorphism
classes and using the previous lemmas (this will be done formally
later), the only estimate left is the following one. This is a main
point of departure from Bordenave-Collins and replaces \cite[Lemma 26]{BordenaveCollins}. 
\begin{prop}
\label{prop:sum-over-iso-class}For any compact subset $\K\subset\{\,s\,:\,\Re(s)>\frac{\delta}{2}\,\}$,
there is $\rho=\rho(\K)<1$ such that for all $s\in\K$, all $m>0$
and $\gamma^{0}\in W_{\ell,m}(v,e)$ 
\[
\sum_{\substack{\gamma\text{ isomorphic to \ensuremath{\gamma^{0}} }\\
\gamma=(\gamma_{1},\ldots,\gamma_{2m})
}
}\left|\tr\left[A(s)_{\hat{\i}(\gamma_{1})}A(s)_{\hat{\i}(\gamma_{2})}^{*}\cdots A(s)_{\hat{\i}(\gamma_{2m-1})}A(s)_{\hat{\i}(\gamma_{2m})}^{*}\right]\right|\leq C^{m+\rank+e_{1}}n^{v}\rho^{2\ell m}.
\]
where $e_{1}$ is the number of edges of $G_{\gamma^{0}}$ of multiplicity
one and $\rank=e-v+1$.
\end{prop}

\begin{proof}
Write $s=r+it$. First, we use Lemma \ref{lem:Trace-of-product-bound}
to get
\begin{align}
 & \sum_{\substack{\text{\ensuremath{\gamma\ }isomorphic to \ensuremath{\gamma^{0}} }\\
\gamma=(\gamma_{1},\ldots,\gamma_{2m})
}
}\left|\tr\left[A(s)_{\hat{\i}(\gamma_{1})}A(s)_{\hat{\i}(\gamma_{2})}^{*}\cdots A(s)_{\hat{\i}(\gamma_{2m-1})}A(s)_{\hat{\i}(\gamma_{2m})}^{*}\right]\right|\nonumber \\
 & \leq C^{m}\sum_{\substack{\text{\ensuremath{\gamma\ }isomorphic to \ensuremath{\gamma^{0}} }\\
\gamma=(\gamma_{1},\ldots,\gamma_{2m})
}
}\Upsilon_{\hat{\i}(\gamma_{1})}^{r}\Upsilon_{\hat{\i}(\gamma_{2})}^{r}\cdots\Upsilon_{\hat{\i}(\gamma_{2m})}^{r}\nonumber \\
 & \leq C_{1}^{m}\sum_{\substack{\text{\ensuremath{\gamma\ }isomorphic to \ensuremath{\gamma^{0}} }\\
\gamma=(\gamma_{1},\ldots,\gamma_{2m})
}
}\Upsilon_{\i(\gamma_{1})'}^{r}\Upsilon_{\i(\gamma_{2})'}^{r}\cdots\Upsilon_{\i(\gamma_{2m})'}^{r}\label{eq:keyprop-first-boun}
\end{align}
 for some $C_{1}(\K)>0$, where the last inequality used (\ref{eq:trivial-child-ineq}),
$\Upsilon_{i}\geq c(\Gamma)>0$ for $i\in[2d]$, and (\ref{eq:upsilon-mult-eq-1})
to deduce
\[
\Upsilon_{\hat{\i}}\leq c^{-1}\Upsilon_{\hat{\i}}\Upsilon_{i_{1}}\leq c^{-1}K\Upsilon_{\i}\leq c^{-1}K\Upsilon_{\i'}.
\]

Here we view $\gamma^{0}$ as fixed; let 
\[
\gamma^{0}=(\gamma_{1}^{0},\ldots,\gamma_{2m}^{0}),\quad\gamma_{j}^{0}=((x_{j,1}^{0},i_{j,1}^{0}),\ldots,(x_{j,\ell+1}^{0},i_{j,\ell+1}^{0})).
\]
We build a graph $\tilde{G}$ as follows. Let $\tilde{V}$ denote
the vertices of $G_{\gamma^{0}}$ that are either of degree $\geq3$
or of the form $x_{j,1}^{0}$ or $x_{j,\ell+1}^{0}$ for $j\in[2m]$.
Now, every vertex of $G_{\gamma^{0}}$ that is not in $\tilde{V}$
has degree $2$ and for each of these vertices, we remove it by merging
the two adjacent edges into one edge; this can be done sequentially.
The resulting graph on $\tilde{V}$ is called $\tilde{G}$ and has
edge set that we call $\tilde{E}$. Every edge of $\tilde{E}$ is
labeled by a sequence of edges of $G_{\gamma^{0}}$. We write $e'$
for the number of vertices of $\tilde{G}$ and $v'$ for the number
of vertices. Recall that $\rank\eqdf e-v+1$.

The following (in)equalities appear in \cite[proof of Lemma 26]{BordenaveCollins}
and are either obvious or not hard to check directly:
\begin{align}
e & =v+\rank-1,\label{eq:rank-rewritten}\\
e' & =v'+\rank-1,\nonumber \\
e' & \leq3\cdot\rank+4m-3.\label{eq:e'-ineq}
\end{align}

Choosing $\gamma$ isomorphic to $\gamma^{0}$ amounts to the following
choices:
\begin{itemize}
\item Choosing distinct numbers in $n$ for the vertices of $G_{\gamma^{0}}$,
and, independently,
\item Choosing `allowable' labels in $[2d]$ for the edges of $E_{\gamma^{0}}$.
\end{itemize}
There are clearly $(n)(n-1)\cdots(n-v+1)\leq n^{v}$ of the first
kind of choice. We will later incorporate this into our final estimate.
Now we count the second kind of choice.

The reason for manipulating (\ref{eq:keyprop-first-boun}) into its
current form is that the sequences 
\begin{equation}
\i(\gamma_{1})',\i(\gamma_{2})',\ldots,\i(\gamma_{2m})'\label{eq:long-sequence}
\end{equation}
appearing therein are exactly the sequence of edge colors read by
walking in $G_{\gamma}$ according to the respective $\gamma_{j}$.
Note that by assumption, the underlying graph of $G_{\gamma}$ (forgetting
edge colors) is the same as that of $G_{\gamma_{0}}$ and hence has
underlying topological graph $\tilde{G}$. For each edge of $\tilde{G}$,
the number of times it is \emph{covered} by $\gamma$ is the number
of times $\gamma$ traverses it in either direction. (Note that $\gamma$
and $\gamma_{0}$ cover each edge of $\tilde{G}$ the same number
of times, and in the same orders and directions.)

We now partition \textbf{separately} each $\i(\gamma_{j})'$ in (\ref{eq:keyprop-first-boun})
into subsequences of three\footnote{In Bordenave-Collins type \textbf{0} and type \textbf{1} are considered
as one type but here it is useful to separate these out.} mutually exclusive types:
\begin{description}
\item [{0)}] a subsequence that arises from $\gamma$ traversing an edge
of $\tilde{G}$ that is only ever covered once by $\gamma_{0}$ and
hence also $\gamma$.
\item [{1)}] a subsequence that arises from $\gamma$ traversing an edge
of $\tilde{G}$ for the first or second time (in any direction) that
is covered at least two times by $\gamma_{0}$ and hence also $\gamma$.
\item [{2)}] a subsequence that arises from $\gamma$ traversing a \textbf{\emph{sequence}}\emph{
}of edges of $\tilde{G}$ where each edge is traversed for at least
the third time (in either direction).
\end{description}
(Here subsequences are repeated according to their multiplicity.)
There are at most $2e'$ subsequences of type \textbf{0 }or type \textbf{1
}and hence also at most $2e'+4m$ subsequences of type \textbf{2 }and
at most $4e'+4m$ subsequences in total.

We use the multiplicative estimate (\ref{eq:upsilon-mult-eq-1}) for
$\Upsilon$ to obtain
\begin{equation}
\Upsilon_{\i(\gamma_{1})'}^{r}\Upsilon_{\i(\gamma_{2})'}^{r}\cdots\Upsilon_{\i(\gamma_{2m})'}^{r}\leq K^{(8e'+8m)r}\left(\prod_{\text{type \textbf{0}\textbf{ }\ensuremath{\u}}}\Upsilon_{\u}^{r}\right)\left(\prod_{\text{type \textbf{1}\textbf{ }\ensuremath{\u}}}\Upsilon_{\u}^{r}\right)\left(\prod_{\text{type \textbf{2 }\ensuremath{\u}}}\Upsilon_{\u}^{r}\right).\label{eq:splttingtotypes}
\end{equation}
We deal with these factors in turn.

Write $\{\U_{i}\}$ for the type \textbf{0 }subsequences of $\gamma$.
Their total length is clearly $e_{1}$ and there are at most $e'$
of them.

The subsequences of type \textbf{1 }come in pairs $\u,\v$ of the
form either 
\[
\u=\v,\quad\text{or }\u=\bar{\v}.
\]
In either case the contribution of this pair to $\left(\prod_{\text{type \textbf{1}\textbf{ }\ensuremath{\u}}}\Upsilon_{\u}^{r}\right)$,
by the mirror estimate (\ref{eq:upsilon-mirror-eq}), is 
\[
\leq K^{r}\Upsilon_{\u}^{2r}.
\]
Let $\{\u_{i}\}$ denote a collection of representative subsequences
of type \textbf{1}, one for each pair of type \textbf{1. }The total
number of these pairs is $\leq e'$. The total length of the $\u_{i}$
is $e-e_{1}.$ Therefore
\begin{equation}
\prod_{\text{type \textbf{1}\textbf{ }\ensuremath{\u}}}\Upsilon_{\u}^{r}\leq\prod_{i}\left(K^{r}\Upsilon_{\u_{i}}^{2r}\right)=K^{e'r}\prod_{i}\Upsilon_{\u_{i}}^{2r}.\label{eq:type1-pointwise}
\end{equation}

Finally we deal with the type \textbf{2 }subsequences. Recalling that
there are at most $2e'+4m$ of these and their total length is $2(\ell m-e)+e_{1}$
as in \cite[(44)]{BordenaveCollins}, by using (\ref{eq:upsilon-exp-bound})
we obtain
\begin{equation}
\prod_{\text{type \textbf{2}\textbf{ }\ensuremath{\u}}}\Upsilon_{\u}^{r}\leq\prod_{\text{type \textbf{2}\textbf{ }\ensuremath{\u}}}K^{r}D^{-r|\u|}\leq K^{r(2e'+4m)}D^{-r\left(2(\ell m-e\right)+e_{1})}.\label{eq:type2-pointwise-bound}
\end{equation}

Now, to choose $\gamma$ isomorphic to $\gamma_{0}$, we need to choose
the vertex labels in $[n]$ (of which there are $\leq n^{v}$ choices)
and then
\begin{itemize}
\item Choose the sequence of edge colors in each type \textbf{0 }subsequence
$\U_{i}$
\item Choose the sequence of edge colors in each type \textbf{1 }subsequence
$\u_{i}$
\item (This completes the choice of $\gamma$ because type \textbf{2 }subsequences
are determined by the previous choices.)
\end{itemize}
Now combining (\ref{eq:keyprop-first-boun}), (\ref{eq:long-sequence}),
(\ref{eq:splttingtotypes}), (\ref{eq:type1-pointwise}), (\ref{eq:type2-pointwise-bound})
we obtain
\begin{align}
 & \sum_{\substack{\text{\ensuremath{\gamma\ }isomorphic to \ensuremath{\gamma^{0}} }\\
\gamma=(\gamma_{1},\ldots,\gamma_{2m})
}
}\left|\tr\left[A(s)_{\i(\gamma_{1})}A(s)_{\i(\gamma_{2})}^{*}\cdots A(s)_{\i(\gamma_{2m-1})}A(s)_{\i(\gamma_{2m})}^{*}\right]\right|\nonumber \\
 & \leq n^{v}C_{1}^{m}K^{(11e'+12m)r}D^{-r\left(2(\ell m-e\right)+e_{1})}\left(\prod_{i}\Upsilon_{\U_{i}}^{r}\right)\left(\prod_{i}\Upsilon_{\u_{i}}^{2r}\right).\label{eq:pre-pressure}
\end{align}
Now using the bound in terms of the pressure from Lemma \ref{lem:Pressure-estimate}
we obtain that 
\begin{align*}
\eqref{eq:pre-pressure} & \leq n^{v}C_{1}^{m}K^{(11e'+12m)r}D^{-r\left(2(\ell m-e\right)+e_{1})}C^{2e'}\text{\ensuremath{\left(\prod_{i}\exp(|\U_{i}|P(r_{1}))\right)}}\left(\prod_{i}\exp(|\u_{i}|P(2r_{1}))\right)\\
 & =n^{v}C_{1}^{m}K^{(11e'+12m)r}D^{-r\left(2(\ell m-e\right)+e_{1})}C^{2e'}\exp(e_{1}P(r_{1}))\exp((e-e_{1})P(2r_{1}))
\end{align*}
where 
\[
r_{1}\eqdf\min\{\,\Re(s)\,:\,s\in\K\,\}>\frac{\delta}{2}.
\]
We conclude by using (\ref{eq:rank-rewritten}) and (\ref{eq:e'-ineq})
to obtain for some consolidated constant $c>1$ depending on all parameters
\[
\eqref{eq:pre-pressure}\leq c^{m+\rank+e_{1}}n^{v}D^{-r_{1}2(\ell m-v)}\exp(vP(2r_{1})).
\]
The key point here is that $P(2r_{1})<0$.

If $v\leq\frac{\ell m}{2}$ we get the stated result with $\rho=D^{-\frac{r_{1}}{2}}$.
Otherwise $v>\frac{\ell m}{2}$ and we get the result with $\rho=\exp\left(\frac{P(2r_{1})}{4}\right)<1$.
In any case we get the result with $\rho<1$ equal to the maximum
of these two values. This concludes the proof.
\end{proof}
\begin{proof}[Proof of Proposition \ref{prop:B(ell)-norm-bound}]

Given Lemmas \ref{lem:iso-class-count-B} and \ref{lem:BC-random-matrix-1}
and Proposition \ref{prop:sum-over-iso-class} the proof is very similar
to \cite[proof of Prop. 24]{BordenaveCollins} but we give the details
here for completeness.

For $n\geq3$ let 
\[
m=\left\lfloor \frac{\log n}{13\log\log n}\right\rfloor .
\]

We partition all possible isomorphism classes of paths $\gamma^{0}$
in $W_{\ell,m}$ according to the number $v$ of vertices and number
$e$ of edges of $G_{\gamma^{0}}$ and also the number $e_{1}$ of
edges of $G_{\gamma^{0}}$ that are multiplicity one. We write $\rank\eqdf e-v+1.$
We have the inequalities
\begin{equation}
2(e-\ell m)\leq e_{1}\leq e\label{eq:e1-ineq}
\end{equation}
 by \cite[Lemma 27]{BordenaveCollins} and can assume $e-v+1\geq0$
(or else $W_{\ell,m}(v,e)$ is empty).

The contribution to (\ref{eq:exp-bl-bound-1}) from a given isomorphism
class with parameters $v,e,e_{1}$ is 
\begin{equation}
\leq n(cC)^{m+\rank}\left(\frac{1}{n}\right)^{\rank}\left(\frac{6\ell m}{\sqrt{n}}\right)^{\max\left(e_{1}-4\cdot\rank-4m,0\right)}C^{e_{1}}\rho^{2\ell m}\label{eq:iso-class-contrib-bound}
\end{equation}
by Lemma \ref{lem:BC-random-matrix-1} and Proposition \ref{prop:sum-over-iso-class}.

If $e_{1}-4\cdot\rank-4m\leq0$ then $C^{e_{1}}\leq(C^{4})^{\rank+m}$
and therefore
\[
\eqref{eq:iso-class-contrib-bound}\leq n(cC^{5})^{m+\rank}\left(\frac{1}{n}\right)^{\rank}\left(\frac{6\ell m}{\sqrt{n}}\right)^{0}\rho^{2\ell m}.
\]
If alternatively $e_{1}-4\cdot\rank-4m\geq0$ then by (\ref{eq:e1-ineq})
we get
\begin{align*}
\eqref{eq:iso-class-contrib-bound} & \leq n(cC^{5})^{m+\rank}\left(\frac{1}{n}\right)^{\rank}\left(\frac{6C\ell m}{\sqrt{n}}\right)^{e_{1}-4\cdot\rank-4m}\rho^{2\ell m}\\
 & \leq n(cC^{5})^{m+\rank}\left(\frac{1}{n}\right)^{\rank}\left(\frac{6C\ell m}{\sqrt{n}}\right)^{2e-2\ell m-4\cdot\rank-4m}\rho^{2\ell m}\\
 & =n(cC^{5})^{m+\rank}\left(\frac{1}{n}\right)^{\rank}\left(\frac{6C\ell m}{\sqrt{n}}\right)^{2v-2(\ell+2)m-2\cdot\rank}\rho^{2\ell m}.
\end{align*}
In any case we get that the contribution to (\ref{eq:exp-bl-bound-1})
from a single isomorphism class is 
\[
\leq n(cC^{5})^{m+\rank}\left(\frac{1}{n}\right)^{\rank}\left(\frac{6C\ell m}{\sqrt{n}}\right)^{\max(2v-2(\ell+2)m-2\cdot\rank,0)}\rho^{2\ell m}.
\]
Now, using Lemma \ref{lem:iso-class-count-B} and summing over $v$
and $\rank$ we get 
\begin{align*}
 & \E[\|\uB(s)^{(\ell)}\|^{2m}]\\
\leq & n\rho^{2\ell m}\sum_{v=1}^{\infty}\sum_{\rank=0}^{\infty}\left(4d\ell m\right)^{6m\cdot\rank+10m}(cC^{5})^{m+\rank}\left(\frac{1}{n}\right)^{\rank}\left(\frac{6C\ell m}{\sqrt{n}}\right)^{\max(2v-2(\ell+2)m-2\cdot\rank,0)}.
\end{align*}
Now a series of elementary bounds involving geometric series that
can be found in \cite[proof of Prop. 24]{BordenaveCollins} leads
to
\[
\E[\|\uB(s)^{(\ell)}\|^{2m}]\leq n\rho^{2\ell m}(c'\ell m)^{10m}
\]
for some new constant $c'=c'(\K)>0$. Let $\rho'>\rho$ be a value
in $(\rho,1)$.

Applying Markov's inequality now gives 
\begin{align*}
\mathrm{Prob}[\|\uB(s)^{(\ell)}\|\geq(\log n)^{20}(\rho')^{\ell}] & \leq\frac{n(c'\ell m)^{10m}}{(\log n)^{40m}}\left(\frac{\rho}{\rho'}\right)^{2\ell m}\\
 & \leq\frac{n(c'\log n)^{20m}}{(\log n)^{40m}}\left(\frac{\rho}{\rho'}\right)^{2\ell m}\leq C'\left(\frac{\rho}{\rho'}\right)^{2\ell m}
\end{align*}
as required.
\end{proof}

\subsection{Norm of $R_{k}(s)^{(\ell)}$\label{sec:Norm-of-Rk(ell)}}
\begin{prop}
\label{prop:Rkl-norm-bound}Let $\K$ be a compact subset of $\{s\,:\,\Re(s)>\frac{\delta}{2}\,\}.$
There is $c=c(\K)>0$ and $\rho_{1}=\rho_{1}(\K)>1$ such that for
any $s\in\K$, for all $1\leq k\leq\ell\leq\log n$ the event
\[
\|R_{k}(s)^{(\ell)}\|\leq(\log n)^{40}\rho_{1}^{\ell}
\]
 holds with probability at least $1-c\exp\left(-\frac{\ell\log n}{c\log\log n}\right)$.
\end{prop}

First we follow \cite[\S\S\S 4.4.2]{BordenaveCollins} and define\footnote{Here there is a typo in \cite{BordenaveCollins}; we checked with
the authors that they meant to define $\hat{W}_{\ell,m,k}$ (which
Bordenave-Collins call $\hat{W}_{\ell,m}$) using $F_{k}^{\ell+1}\backslash F^{\ell+1}$
(not $F_{k}^{\ell+1}$ as written there). } $\hat{W}_{\ell,m,k}$ to be the set of $\gamma=(\gamma_{1},\ldots,\gamma_{2m})$
satisfying the same conditions as the elements of $W_{\ell,m}$ \textbf{except}
\textbf{here}, each $\gamma_{i}$ is required to be in $F_{k}^{\ell+1}\backslash F^{\ell+1}$,
recalling the definition of $F_{k}^{\ell+1}$ and $F^{\ell+1}$ from
Definition \ref{def:tangles}. 

We then obtain as in \cite[pg 851]{BordenaveCollins} 
\begin{align}
\|R_{k}(s)^{(\ell)}\|^{2m}\leq & \tr\left(\left(R_{k}(s)^{(\ell)}(R_{k}(s)^{(\ell)})^{*}\right)^{m}\right)\nonumber \\
= & \sum_{\substack{\gamma=(\gamma_{1},\ldots,\gamma_{2m})\in\hat{W}_{\ell,m,k}\\
\gamma_{j}=((x_{j,1},i_{j,1}),\ldots,(x_{j,\ell+1},i_{j,\ell+1}))
}
}\tr\left[A(s)_{\hat{\i}(\gamma_{1})}A(s)_{\hat{\i}(\gamma_{2})}^{*}\cdots A(s)_{\hat{\i}(\gamma_{2m-1})}A(s)_{\hat{\i}(\gamma_{2m})}^{*}\right]\nonumber \\
 & \cdot\prod_{j=1}^{2m}\left(\prod_{t=1}^{k-1}(\underline{S}_{i_{j,t}})_{x_{j,t}x_{j,t+1}}\right)\left(\text{\ensuremath{\prod_{t=k+1}^{\ell}}(\ensuremath{S_{i_{j,t}})_{x_{j,t}x_{j,t+1}}}}\right).\label{eq:matrix-product}
\end{align}
Isomorphism classes of elements of $\hat{W}_{\ell,m,k}$ are defined
the same way as for isomorphism classes in $W_{\ell,m}.$ 

For $\gamma\in\hat{W}_{\ell,m,k}$ we define $\hat{G}_{\gamma}$ to
be the following graph. If we write 
\[
\gamma=(\gamma_{1},\ldots,\gamma_{2m}),\quad\gamma_{j}=((x_{j,1},i_{j,1}),\ldots,(x_{j,\ell+1},i_{j,\ell+1}))
\]
then the vertex set of $\hat{G}_{\gamma}$ is defined to be the set
of all $x_{j,t}$ with $j\in[2m],t\in[\ell+1]$ and the colored edges
\[
\{[x_{j,t},i_{j,t},x_{j,t+1}]\,:\,1\leq t\leq k-1,\,k+1\leq t\leq\ell\}
\]
modulo the equivalence relation (\ref{eq:equiv-relation}). Notice
this is a sub-colored-graph of $G_{\gamma}$ with restriction on the
edges corresponding to the matrix product in (\ref{eq:matrix-product}).

For $\gamma\in\text{\ensuremath{\hat{W}_{\ell,m,k}} }$ it is possible
for $\hat{G}_{\gamma}$ to be disconnected but, as in \cite[(53)]{BordenaveCollins},
$\hat{G}_{\gamma}$ cannot have more vertices than edges; this is
actually the purpose of ensuring each $\gamma_{i}$ is tangled. We
have from \cite[Lemma 29]{BordenaveCollins} the following result.
\begin{lem}
\label{lem:R-num-iso-classes}Let $\hat{W}_{\ell,m,k}(v,e)$ denote
the elements of $\hat{W}_{\ell,m,k}$ such that $\hat{G}_{\gamma}$
has $v$ vertices and $e$ edges. The number of isomorphism classes
in $\hat{W}_{\ell,m,k}(v,e)$ is 
\[
\leq\left(2d\ell m\right)^{12m\cdot\rank+20m}
\]
where as before, $\rank\eqdf e-v+1\geq1$ if $\hat{W}_{\ell,m,k}$
is non-empty.
\end{lem}

We also have the following results from \cite[pg. 853]{BordenaveCollins},
the first of which is trivial.
\begin{lem}
\label{lem:Relements-in-iso-class}The number of elements of any isomorphism
class in $\hat{W}_{\ell,m,k}(v,e)$ is at most
\[
n^{v}d^{e}.
\]
\end{lem}

\begin{lem}
\label{lem:R-expected-product-mat-coefs}There is a constant $c>0$
such that if $2\ell m\leq\sqrt{n}$ and $\gamma=(\gamma_{1},\ldots,\gamma_{2m})\in\hat{W}_{\ell,m,k}(v,e),$with
$\gamma_{j}=((x_{j,1},i_{j,1}),\ldots,(x_{j,\ell+1},i_{j,\ell+1}))$
for $j\in[2m]$, then
\[
\left|\E\left[\prod_{j=1}^{2m}\left(\prod_{t=1}^{k-1}(\underline{S}_{i_{j,t}})_{x_{j,t}x_{j,t+1}}\right)\left(\text{\ensuremath{\prod_{t=k+1}^{\ell}}(\ensuremath{S_{i_{j,t}})_{x_{j,t}x_{j,t+1}}}}\right)\right]\right|\leq c\left(\frac{9}{n}\right)^{e}.
\]
\end{lem}

\begin{proof}[Proof of Proposition \ref{prop:Rkl-norm-bound}]
Suppose $n\ge3$ and define 
\[
m=\left\lfloor \frac{\log n}{25\log\log n}\right\rfloor .
\]
Taking the expected value of (\ref{eq:matrix-product}), and using
Lemma \ref{lem:Trace-of-product-bound} and the fact that $\Upsilon_{\i}\leq C$
for any $\i$ we obtain for $C=C(\Lambda)>0$
\begin{align*}
 & \E\left[\|R_{k}(s)^{(\ell)}\|^{2m}\right]\\
\leq & C^{m}\sum_{\substack{\gamma=(\gamma_{1},\ldots,\gamma_{2m})\in\hat{W}_{\ell,m,k}\\
\gamma_{j}=((x_{j,1},i_{j,1}),\ldots,(x_{j,\ell+1},i_{j,\ell+1}))
}
}\left|\E\left[\prod_{j=1}^{2m}\left(\prod_{t=1}^{k-1}(\underline{S}_{i_{j,t}})_{x_{j,t}x_{j,t+1}}\right)\left(\text{\ensuremath{\prod_{t=k+1}^{\ell}}(\ensuremath{S_{i_{j,t}})_{x_{j,t}x_{j,t+1}}}}\right)\right]\right|.
\end{align*}
Now noting that the product of matrix coefficients above is constant
on isomorphism classes in $\hat{W}_{\ell,m,k}(v,e)$, and using Lemmas
\ref{lem:R-num-iso-classes}, \ref{lem:Relements-in-iso-class}, and
\ref{lem:R-expected-product-mat-coefs}, we obtain
\begin{align*}
\E\left[\|R_{k}(s)^{(\ell)}\|^{2m}\right] & \leq cC^{m}\sum_{v=1}^{2\ell m}\sum_{e=v}^{\infty}\left(2d\ell m\right)^{12m\cdot(e-v+1)+20m}\left(\frac{9}{n}\right)^{e}n^{v}d^{e}\\
 & =(C'd\ell m)^{32m}\sum_{v=1}^{2\ell m}\left(\frac{n}{(2d\ell m)^{12m}}\right)^{v}\sum_{e=v}^{\infty}\left(\frac{9d}{n}\left(2d\ell m\right)^{12m}\right)^{e}.
\end{align*}
The inner geometric series is convergent by our constraint on $\ell,m$
so we obtain
\begin{align*}
\E\left[\|R_{k}(s)^{(\ell)}\|^{2m}\right] & \leq(C'd\ell m)^{32m}\sum_{v=1}^{2\ell m}\left(\frac{n}{(2d\ell m)^{12m}}\frac{9d}{n}\left(2d\ell m\right)^{12m}\right)^{v}\\
 & =(C'd\ell m)^{32m}\sum_{v=1}^{2\ell m}(9d)^{v}\\
 & \leq(C''d\ell m)^{32m}(9d)^{2\ell m}.
\end{align*}
Using Markov's inequality yields 
\begin{align*}
\mathrm{Prob}[\|R_{k}(s)^{(\ell)}\|\geq(\log n)^{40}(10d)^{\ell}] & \leq\frac{(C''d\ell m)^{32m}}{(\log n)^{80m}}\left(\frac{9}{10}\right)^{2\ell m}\\
 & \leq\frac{(c(\log n))^{64m}}{(\log n)^{80m}}\left(\frac{9}{10}\right)^{2\ell m}\\
 & \leq c'\left(\frac{9}{10}\right)^{2\ell m}.
\end{align*}
This directly implies the result.
\end{proof}

\section{Proof of Theorem \ref{thm:main-theorem}}

It is sufficient to prove Theorem \ref{thm:main-theorem} when $\K$
is a rectangle in $\{\,s\,:\,\Re(s)\,>\,\frac{\delta}{2}\,\}$ (by
covering $\K$ with finitely many rectangles). So suppose $\K$ is
such a rectangle. Let $\rho$ and $\rho_{1}$ be the constants provided
by Propositions \ref{prop:B(ell)-norm-bound} and \ref{prop:Rkl-norm-bound}
for this $\K$. Let $C_{0}$ be the constant provided by Lemma \ref{lem:norm-of-variation}
for this $\K$.

We let $ $
\begin{align}
\beta & =\min\left(\frac{1}{8\log(2d-1)},\log\left(\frac{\rho_{1}}{\rho}\right)^{-1}\right),\label{eq:beta-choice}\\
\alpha & =2\beta\log C_{0}.\label{eq:alpha-choice}
\end{align}

A $\delta$-net of $\K$ is a finite subset of $\K$ such that any
point of $\K$ is within Euclidean distance $\delta$ of some point
of the $\delta$-net. Now it is easy to see that for each $n$ we
can choose an $n^{-\alpha}$-net $\mathcal{N}_{n}$ of $\K$ with
\[
|\mathcal{N}_{n}|\leq Cn^{2\alpha}
\]
 for some $C=C(\K,\alpha)>0$. We want all the following events to
hold simultaneously for 
\[
\ell=\lfloor\beta\log n\rfloor.
\]

\begin{description}
\item [{A}] $G_{\sigma}$ is $\ell$-tangle-free.
\item [{B}] $\|\uB(s_{i})^{(\ell)}\|\leq(\log n)^{20}\rho^{\ell}$ for
$\rho<1$ as in Proposition \ref{prop:B(ell)-norm-bound} \emph{for
all} $s_{i}\in\mathcal{N}_{n}$.
\item [{C}] $\|R_{k}(s_{i})^{(\ell)}\|\leq(\log n)^{40}\rho_{1}^{\ell}$
for $\rho_{1}>0$ as in Proposition \ref{sec:Norm-of-Rk(ell)} \emph{for
all $s_{i}\in\mathcal{N}_{n}$ and all $1\leq k\leq\ell$.}
\end{description}
By Lemma \ref{lem:tangle-probability} and Propositions \ref{prop:B(ell)-norm-bound}
and \ref{sec:Norm-of-Rk(ell)}, these events all hold with probability
at least

\[
1-c\frac{\ell^{3}(2d-1)^{4\ell}}{n}-cCn^{2\alpha}\exp\left(-\frac{\ell\log n}{c\log\log n}\right)
\]
for some $c>0$. The last term above tends to zero for any fixed $\alpha,\beta$
and the second term tends to zero by our choice of $\beta$ in (\ref{eq:beta-choice}).
So with probability tending to one as $n\to\infty$, all three families
\textbf{A,~B, C }of events above hold.

Now, if \textbf{A, B, C} hold then by Lemma \ref{lem:Bell-to-B(ell)-and-Rkl}
\begin{align*}
\|B(s_{i})^{\ell}\lvert_{K_{0}}\|\leq & \|\uB(s_{i})^{(\ell)}\|+\frac{1}{n}\sum_{k=1}^{\ell}\|R_{k}(s_{i})^{(\ell)}\|\\
\leq & (\log n)^{20}\rho^{\ell}+\frac{\ell}{n}(\log n)^{40}\rho_{1}^{\ell}\\
\leq & (1+\beta)(\log n)^{41}\rho^{\ell}
\end{align*}
for all $s_{i}\in\mathcal{N}_{n}$, where the last inequality used
(\ref{eq:beta-choice}). Therefore by Proposition \ref{prop:conjugation}
we have 
\begin{equation}
\|\L_{s_{i},\rho_{n}^{0}}^{\ell}\|\leq(1+\beta)(\log n)^{41}\rho^{\ell}\label{eq:contraction-on-net}
\end{equation}
for all $s_{i}\in\mathcal{N}_{n}$.

Now by Lemma \ref{lem:norm-of-variation} there is a constant $C_{0}=C_{0}(K)>0$
such that for all $s\in\K$, if $s_{i}$ is a point in $\mathcal{N}_{n}$
for distance $\leq n^{-\alpha}$ from $s$ then

\begin{equation}
\|\L_{s,\rho_{n}^{0}}^{\ell}-\L_{s_{0},\rho_{n}^{0}}^{\ell}\|\leq|s-s_{0}|C_{0}^{\ell}\leq n^{-\alpha}C_{0}^{\ell}\leq n^{\left(\beta\log C_{0}-\alpha\right)}<n^{-\frac{\alpha}{2}}\label{eq:variation-from-net}
\end{equation}
by our choice of $\alpha$ in (\ref{eq:alpha-choice}).

Combining (\ref{eq:contraction-on-net}) and (\ref{eq:variation-from-net})
we obtain that with probability tending to one as $n\to\infty$, for
all $s\in\K$
\[
\|\L_{s,\rho_{n}^{0}}^{\ell}\|<1.
\]
This immediately implies that $\L_{s,\rho_{n}^{0}}$ does not have
$1$ as an eigenvalue, and hence by Proposition \ref{prop:new-resonances-and-evalue-1}
there are no new resonances of the random cover $X_{\phi}$ in $\K$.
$\square$

\bibliographystyle{alpha}
\bibliography{database}

\newpage{}

\noindent Michael Magee, \\
Department of Mathematical Sciences,\\
Durham University, \\
Lower Mountjoy, DH1 3LE Durham,\\
United Kingdom

\noindent \texttt{michael.r.magee@durham.ac.uk}\\

\noindent Frédéric Naud, \\
Institut Mathématique de Jussieu,\\
Sorbonne Université,\\
4 place Jussieu, 75252 Paris Cedex 05\\
France\\
\texttt{frederic.naud@imj-prg.fr}
\end{document}